\documentclass[preprint,11pt]{elsarticle}

\usepackage{graphicx}
\usepackage{amssymb}
\usepackage{empheq}
\usepackage{amsthm}
\usepackage{amsmath, amsfonts}
\usepackage{float}
\usepackage{subcaption}
\usepackage{multirow}
\usepackage{makecell}
\usepackage{caption}
\usepackage{enumitem}
\usepackage{color}
\usepackage{soul}
\usepackage{empheq}
\usepackage{multicol}
\usepackage{algorithm}
\usepackage[noend]{algpseudocode}
\usepackage{setspace}
\usepackage{mathtools}
\usepackage{comment}
\usepackage{bm}
\usepackage{amsfonts}
\usepackage{epstopdf}
\usepackage{standalone}
\usepackage[colorlinks]{hyperref}
\usepackage{ulem}
\usepackage{cleveref}
\usepackage{todonotes}

\usepackage{relsize}
\usepackage{tikz}
\usetikzlibrary{shapes}
\usetikzlibrary{plotmarks}
\usepackage{csvsimple}
\usepackage{filecontents}

\usepackage{mathrsfs}

\newtheorem{thm}{Theorem}[section]
\newtheorem{lemma}[thm]{Lemma}

\newtheorem{remark}[thm]{Remark}

\newcommand{\beq}{\begin{equation}}
\newcommand{\eeq}{\end{equation}}

\newcommand{\ep}{{\epsilon}}

\newcommand{\del}{\delta}

\def\cI{\mathcal{I}}

\def\cI{\mathcal{I}}

\def\R{\mathbb{R}}
\def\Z{\mathbb{Z}}

\def\nd{{\textnormal{d}}}

\newcommand{\tcr}[1]{\textcolor{red}{#1}}
\newcommand{\tcg}[1]{\textcolor{green!50!black}{#1}}

\journal{Editor}

\begin{document}

\begin{frontmatter}



\title{Asymptotically compatible reproducing kernel collocation and meshfree integration for the peridynamic Navier equation}

\author[1]{Yu Leng\corref{cor1}}
\ead{yu-leng@utexas.edu}
\author[2]{Xiaochuan Tian}
\author[3]{Nathaniel A. Trask\fnref{fn1}}
\author[1]{John T. Foster}

\address[1]{Department of Petroleum and Geosystems Engineering, The University of Texas at Austin, Austin, TX 78712, United States}
\address[2]{Department of Mathematics, The University of Texas at Austin, Austin, TX 78712, United States}
\address[3]{Center for Computing Research, Sandia National Laboratories, Albuquerque, NM 87123, United States}

\cortext[cor1]{Corresponding author}
\fntext[fn1]{Sandia National Laboratories is a multimission laboratory managed and operated by National Technology and Engineering Solutions of Sandia, LLC., a wholly owned subsidiary of Honeywell International, Inc., for the U.S. Department of Energy’s National Nuclear Security Administration under contract DE-NA-0003525. This paper describes objective technical results and analysis. Any subjective views or opinions that might be expressed in the paper do not necessarily represent the views of the U.S. Department of Energy or the United States Government.}

\begin{abstract}
In this work, we study the reproducing kernel (RK) collocation method for the peridynamic Navier equation. We first apply a linear RK approximation on both displacements and dilatation, then back-substitute dilatation, and solve the peridynamic Navier equation in a pure displacement form. The RK collocation scheme converges to the nonlocal limit and also to the local limit as nonlocal interactions vanish. The stability  is shown by comparing the collocation scheme with the standard Galerkin scheme using Fourier analysis.  We then apply the RK collocation to the quasi-discrete peridynamic Navier equation and show its convergence to the correct local limit when the ratio between the nonlocal length scale and the discretization parameter is fixed.  The analysis is carried out on a special family of rectilinear Cartesian grids for the RK collocation method with a designated kernel with finite support. We assume the Lam\'{e} parameters satisfy $\lambda \geq \mu$ to avoid adding extra constraints on the nonlocal kernel.  Finally, numerical experiments are conducted to validate the theoretical results.

\end{abstract}


\end{frontmatter}


\section{Introduction}
Peridynamics is a nonlocal theory of continuum mechanics introduced by Silling in \cite{Silling2000,Silling2007}. Peridynamic models avoid the use of spatial differentiation and they have attracted interest among researchers, especially for treating problems with fractures and material failure \cite{Bobaru2012, Ha2010, Ouchi2017}. Mathematical analysis of the peridynamics models have been carried out in \cite{Marta2017,Du2013PD,Du2013nonlocal,Mengesha2014, mengesha2014nonlocal} and it is well understood that the linear peridynamic Naiver equation is well-posed. Many numerical methods have been developed to solve the peridynamic Naiver equation \cite{bobaru2016handbook, du2013posteriori,  Macek2007, Pasetto2018, Seleson2016, Seleson2016C, Silling2005, Tian2014, Trask2019} and this is the main focus of our work. Other than \cite{Tian2014} which is a variational method, the rest solve the nonlocal governing equation in the strong form but fail to show rigorous analysis. To our knowledge, this is the first work that provides convergence analysis for solving the strong form of the peridynamic Navier equation.

Nonlocal models introduce a length scale $\delta$, called the \textit{horizon} in peridynamics, which takes into account interactions over finite distances. As $\delta \to 0$, the nonlocal interactions vanish and the nonlocal model recovers its local limit, i.e., a partial differential equation.  Numerical methods that preserve this limiting behaviour in discrete form are called asymptotically compatible (AC) schemes \cite{Tian2013, Tian2014}; many numerical methods for nonlocal models are not AC and may converge to the wrong local limit \cite{Tian2013}. It is a challenge to design AC numerical schemes for nonlocal models. 
Another difficulty is accurate evaluation of the nonlocal integral, which can be computationally prohibitive especially when the nonlocal kernels are singular, and it is often necessary to use a high-order Guassian quadrature rule \cite{XChen2011,Pasetto2018}. Many works have been done to address these two challenges \cite{Yang2018, seleson2014improved,Seleson2016C,Tian2014,Trask2019,Yu2011}.

The Finite Element Method (FEM) \cite{Tian2014} with linear basis functions is AC but the evaluation of the double integral \cite{XChen2011} (in one-dimension) discourages the use of the variational formulation for nonlocal models.  Many mesh-free methods \cite{seleson2014improved, Seleson2016C,Silling2005} for peridynamics, which use the volume of the particles as integration weights, are easy to implement but these methods do not converge to the correct local limit as the nonlocal length scale vanishes. 
A mesh-free integration scheme for the peridynamic Navier equation is introduced in \cite{Trask2019}, however, it lacks convergence analysis. The quadrature weights are calculated using the generalized moving least square technique and this mesh-free integration scheme converges to the correct local limit for nonlocal diffusion \cite{Leng2019b}.
An important consequence of the mesh-free integration scheme is that it is straightforward to include ``bond breaking'' which provides a way to simulate fractures or material failure \cite{Trask2019}.   

We have developed an AC RK collocation scheme for nonlocal diffusion models and introduced a quasi-discrete nonlocal diffusion operator using a mesh-free integration technique \cite{Leng2019b} to avoid using high-order Gauss quadrature rules and save the computational costs. RK collocation on this quasi-discrete nonlocal diffusion operator converges to the correct local limit. The purpose of this work is to extend the methodology to the peridynamic Navier equation.

First, we show RK collocation on the peridynamic Navier equation is AC. We use a similar strategy as in \cite{Leng2019b} to show the stability and consistency of the RK collocation method. The key idea for stability analysis is to compare the Fourier representation of the collocation scheme with the Galerkin scheme \cite{Costabel1992,Leng2019b}; this idea has also appeared in \cite{arnold1984asymptotic, arnold1983asymptotic}. Since the Fourier symbol of the  peridynamic Navier operator is a matrix instead of a scalar, the stability analysis is more involved for the peridynamic Navier equation than that of the nonlocal diffusion. Indeed, in order to simplify the discussion, we need to assume that the two Lam\'e parameters, $\lambda$ and $\mu$, satisfy the constraint $\lambda\geq \mu$.  The uniform consistency, which is crucial to show that the scheme is AC,  is established using the synchronized convergence of the linear RK approximation \cite{chen2017meshfree,Li1996, Li1998synchronized}.
Then, to obviate the need to use high-order Gaussian quadrature rules, we introduce the quasi-discrete peridyanmic Navier equation. Convergence analysis of the RK collocation scheme on the quasi-discrete peridynamic Navier operator is presented when the ratio between horizon $\delta$ and the grid size $h_{\max}$ is fixed.

This paper is organized as follows. In  \cref{sec:ModelEquation}, we introduce the peridynamic Navier equation with Dirichlet boundary conditions and also the quasi-discrete counterparts using finite summation of symmetric quadrature points to replace the integral. In  \cref{sec:RKCollocation}, we present the RK collocation method with special choices of the RK support size. \Cref{sec:Convergence} discusses the convergence analysis of the RK collocation method for the peridynamic Naiver equation and shows that this RK collocation scheme is AC. Then the convergence analysis of the collocation method on the quasi-discrete  peridynamic Navier equation is presented in \cref{sec:DConvergence}. \Cref{sec:NumericalExample} gives numerical examples to complement our theoretical analysis. Finally, we provide conclusions in \cref{sec:Conclusion}.

\section{Peridynamic Navier equation} \label{sec:ModelEquation}
In this section, we first introduce some notations that are used throughout the paper. The spatial dimension is denoted as $\nd$ $(2 \textnormal{ or } 3) $. An arbitrary point $\bm{x} \in \mathbb{R}^\nd$ is expressed as $\bm{x}=(x_1, \ldots, x_{\nd})$. A multi-index,  $\bm{\alpha} = (\alpha_1, \ldots, \alpha_{\nd})$, is a collection of d non-negative integers and its length is $|\bm{\alpha}| = \sum_{i=1}^{\nd}\alpha_i$. As a consequence, we write $\bm{x}^{\bm{\alpha}}=x_1^{\alpha_1}\ldots x_{\nd}^{\alpha_{\nd}}$ for a given $\bm{\alpha}$. We let $\Omega \subset \mathbb{R}^{\nd}$ be an open bounded domain and then the corresponding interaction domain is defined as  
\[
\Omega_{\cI} = \{ \bm{x} \in \mathbb{R}^{\nd} \backslash \Omega :  \text{dist}(\bm{x}, \Omega) \leq 2 \delta\}\, ,
\] 
where $\delta$ is the nonlocal length scale and we denote $\Omega_{\delta} = \Omega \cup \Omega_{\cI}$.

Next, we present the state-based linearized peridynamic Navier equation introduced in \cite{Silling2007,Silling2010}, then use the quasi-discrete nonlocal operators proposed in \cite{Leng2019b} to formulate the quasi-discrete counterparts. 
We differ in convention of the notations from nonlocal vector calculus \cite{du2013posteriori} which is more suited for variational formulation \cite{Du2013PD}, but use instead the notations for the bond-based peridynamics operator together with the nonlocal divergence and gradient operators as defined in \cite{Tian2018StabNL, Tian2018SPH, lee2019nonlocal}. These notations will alleviate the presentation for collocation method which will be introduced in the next section.

\subsection{Nonlocal operators}
The linearized state-based  peridynamic Navier operator consists of two parts; one is the bond-based peridynamic operator and the other is the composition of the nonlocal gradient and divergence operators. The bond-based peridynamic operator is defined, for a given vector-valued function $\bm{u}(\bm{x}) : \Omega_{\delta} \to \mathbb{R}^\nd$, as 
\begin{equation} \label{eqn:BondOperator} 
\mathcal{L}^B_{\delta} \bm{u}(\bm{x}) = \int_{\Omega_{\delta}} \rho_{\delta}(|\bm{y}-\bm{x}|) \frac{\bm{y}-\bm{x}}{|\bm{y}-\bm{x}|} \otimes\frac{\bm{y}-\bm{x}}{|\bm{y}-\bm{x}|} (\bm{u}(\bm{y}) -\bm{u}(\bm{x})) d\bm{y}, \quad  \forall \bm{x}  \in \Omega, 
\end{equation}
where $\rho_{\delta} (|\bm{y}-\bm{x}|)$ is  the nonlocal kernel. We assume the nonlocal kernel is non-negative and symmetric, and has the following scaling,
\begin{equation} \label{eqn:NonlocalKernelScaling} 
\rho_{\delta}(|\bm{s}|) =  \frac{ 1}{ \delta^{\nd+2} } \rho \left(\frac{|\bm{s}|}{\delta} \right), 
\end{equation}
where $\rho (|\bm{s}|)$ is a non-negative and non-increasing function with compact support in $B_{1}$ (for the rest of the paper, we denote $B_{\delta}$ as $B_{\delta}(\bm{0})$, a ball of radius $\delta$ about $\bm{0}$), and it has a bounded second order moment, i.e.,
\begin{equation}  \label{eqn:BoundedMoment}
\int_{B_{1}}\rho(|\bm{s}|)|\bm{s}|^2 d\bm{s} = {\nd} \, .
\end{equation}
The weighted volume $m(\bm{x})$ is defined as
\begin{equation} \label{eqn:WeightedVolume} 
m(\bm{x}) = \int_{\Omega_{\delta}} \rho_{\delta}(|\bm{y}-\bm{x}|) |\bm{y}-\bm{x}|^2 d\bm{y}, \quad  \forall \bm{x}  \in \Omega.
\end{equation}
From \cref{eqn:WeightedVolume,eqn:NonlocalKernelScaling,eqn:BoundedMoment}, it is easy to see that $m(\bm{x}) = \nd$. We remark that the weighted volume defined here as \cref{eqn:WeightedVolume} is a scaled form of the definition in \cite{Silling2000,  Silling2007, Silling_Lehoucq2008}.
Next, the nonlocal divergence operator $\mathcal{D}_{\delta}$ is defined as \cite{Tian2018StabNL, Tian2018SPH},
\begin{equation} \notag 
\mathcal{D}_{\delta} \bm{u}(\bm{x})  = \int_{\Omega_{\delta}} \rho_{\delta}(|\bm{y}-\bm{x}|) (\bm{y}-\bm{x}) \cdot (\bm{u}(\bm{y}) + \bm{u}(\bm{x}) ) d\bm{y}, \quad \forall \bm{x}  \in \Omega,
\end{equation}
and in the sense of principle value, $\mathcal{D}_{\delta}$ can also be written as 
\begin{equation} \label{eqn:DivergenceOperator} 
\mathcal{D}_{\delta} \bm{u}(\bm{x}) = \int_{\Omega_{\delta}} \rho_{\delta}(|\bm{y}-\bm{x}|) (\bm{y}-\bm{x}) \cdot (\bm{u}(\bm{y}) - \bm{u}(\bm{x}) ) d\bm{y}, \quad  \forall \bm{x}  \in \Omega.
\end{equation}
Then, nonlocal dilatation $\theta(\bm{x})$ is defined from the nonlocal divergence operator,
\begin{equation} \label{eqn:Dilatation} 
\theta(\bm{x}) = \frac{\nd}{m(\bm{x})} \mathcal{D}_{\delta} \bm{u}(\bm{x}), \quad  \forall \bm{x}  \in \Omega.
\end{equation}
Then, the nonlocal gradient operator $\mathcal{G}_{\delta}$ is defined by
\begin{equation} \label{eqn:GradientOperator} 
\mathcal{G}_{\delta} \theta(\bm{x}) = \int_{\Omega_{\delta}} \rho_{\delta}(|\bm{y}-\bm{x}|) (\bm{y}-\bm{x}) (\theta(\bm{y}) -\theta(\bm{x})) d\bm{y}, \quad \forall  \bm{x}  \in \Omega.
\end{equation}

Finally, we have the linearized state-based peridynamic Navier operator, 
\begin{equation} \label{eqn:PDoperator}
\mathcal{L}^{S}_{\delta} \bm{u}(\bm{x}) = \frac{C_{\alpha} \, \mu }{m(\bm{x})}\mathcal{L}^B_{\delta} \bm{u}(\bm{x}) + \frac{C_{\beta} \, \nd (\lambda - \mu) }{(m(\bm{x}))^2} \mathcal{G}_{\delta}  \mathcal{D}_{\delta} \bm{u}(\bm{x}), \quad \forall \bm{x} \in \Omega,
\end{equation}
where $C_{\alpha}$ and $C_{\beta}$  are scaling parameters which will be given shortly, and $\lambda$ and $\mu$ are Lam\'{e} parameters which are assumed to be constants in this work. The static peridynamic Navier equation with Dirichlet boundary condition can be formulated as
\begin{equation}  \label{eqn:NonlocalEqn}
\begin{cases}
-\mathcal{L}^{S}_{\delta} \bm{u} = \bm{f}, & \textnormal{in } \Omega, \\
 \quad \quad \bm{u} = \bm{0}, & \textnormal{on } \Omega_{\mathcal{I}}.
\end{cases}
\end{equation}
By introducing $\displaystyle p = (\lambda - \mu) \theta$, we can write \cref{eqn:NonlocalEqn} in a mixed form, as follows,
\begin{equation}  \label{eqn:Mixed}
\begin{cases}  
\displaystyle -\frac{C_{\alpha} \, \mu }{m(\bm{x})}\mathcal{L}^B_{\delta} \bm{u}(\bm{x}) - \frac{C_{\beta}}{m(\bm{x})} \mathcal{G}_{\delta} p (\bm{x}) = \bm{f}(\bm{x}), & \bm{x} \in \Omega,  \\[8pt]
\displaystyle \frac{\nd \, ( \lambda - \mu)}{m(\bm{x})} \mathcal{D}_{\delta} \bm{u}(\bm{x}) - p (\bm{x})  = 0,  &\bm{x} \in \Omega, \\
\bm{u}(\bm{x}) = \bm{0},  & \bm{x} \in \Omega_{\mathcal{I}}.
\end{cases}
\end{equation}

The local limit of $\mathcal{L}^S_{\delta}$ is denoted as $\mathcal{L}^S_0$ when $\delta \rightarrow 0$ \cite{Mengesha2014}. We select $C_{\alpha} = 30, C_{\beta} =3$ for three-dimensional linear elasticity and $C_{\alpha} = 16, C_{\beta} =2$ for two-dimensional plane strain, then
\[
\mathcal{L}^S_0 \bm{u}(\bm{x})= \mu \textnormal{div}(\nabla \bm{u}(\bm{x})) + (\mu+\lambda) \nabla \textnormal{div} \bm{u}(\bm{x}), \quad \forall \bm{x} \in \Omega, 
\] 
and \cref{eqn:NonlocalEqn}  becomes  
\begin{equation}  \label{eqn:LocalEqn}
\begin{cases}
-\mathcal{L}^S_{0} \bm{u} = \bm{f}_0, & \textnormal{in } \Omega, \\
 \quad \quad \bm{u} = \bm{0}, & \textnormal{on } \partial \Omega.
\end{cases}
\end{equation}
We define a space on $\Omega_{\delta}$ with zero volumetric constraint on $\Omega_{\mathcal{I}}$,
\begin{equation} \notag 
L_c^2(\Omega_{\delta}) := \{ \bm{u} \in L^2(\Omega_{\delta}) \mid \bm{u} = \bm{0} \textnormal{ on } \Omega_{\mathcal{I}} \}.
\end{equation}
The natural energy space associated with \cref{eqn:NonlocalEqn} is given in \cite{mengesha2014nonlocal} as
\begin{equation} \notag
\mathcal{S}_{\delta} :=\left \{ \bm{u} \in L_c^2(\Omega_{\delta}): \int_{\Omega_{\delta}} \int_{\Omega_{\delta}}\rho_{\delta}(|\bm{y}-\bm{x}|)(\textnormal{Tr}(\mathcal{D}^*\bm{u})(\bm{y},\bm{x}))^2d\bm{y} d\bm{x} < \infty \right \},
\end{equation}
where $\textnormal{Tr}(\mathcal{D}^*\bm{u})$ is the trace of the operator $\mathcal{D}^*$ defined in \cite{Du2013nonlocal, mengesha2014nonlocal} as
\[
\mathcal{D}^*\bm{u}(\bm{y},\bm{x}):=(\bm{u}(\bm{y})-\bm{u}(\bm{x})) \otimes \frac{\bm{y}-\bm{x}}{|\bm{y}-\bm{x}|}.
\]
The static peridynamic Navier equation (\cref{eqn:NonlocalEqn}) is well-posed and uniformly stable as given in the following theorem \cite{mengesha2014nonlocal}. 
\begin{thm} \label{thm:nonlocalmapping}
Assume $\delta \in (0, \delta_0]$ for some $\delta_0 > 0$. The bilinear form $(-\mathcal{L}^S_{\delta}\bm{u}, \bm{u})$ is an inner product and there exists a constant $C > 0$ which depends on $\delta_0$, such that
\begin{equation} \notag 
|(-\mathcal{L}^{S}_{\delta} \bm{u}, \bm{u})| \geq C \| \bm{u} \|^2_{  L^2(\Omega_{\delta}; \, \mathbb{R}^{\nd}) } \,, \quad \forall \, \bm{u} \in \mathcal{S}_{\delta}.
\end{equation}  
\end{thm}

\subsection{Quasi-discrete nonlocal operators} \label{subsec:quasi-nonlocal}
As introduced in \cite{Leng2019b}, we use a finite number of symmetric quadrature points $\bm{s}$ in the horizon to evaluate the integral such that the weighted volume defined in \cref{eqn:WeightedVolume} is exact, 
\begin{equation} \label{eqn:DWeightedVolume}
m(\bm{x}) = \sum\limits_{\bm{s} \in B^{\epsilon}_{\delta}(\bm{0})}  \omega_{\delta}(\bm{s}) \rho_{\delta}(|\bm{s}|) |\bm{s}|^2 =  \nd , \quad \forall \bm{x} \in \Omega,
\end{equation}
where $\omega_{\delta}(\bm{s})$ is the quadrature weight at the quadrature point $\bm{s}$ and $ B^{\epsilon}_{\delta}(\bm{0}) $ is a finite collection of symmetric quadrature points $\bm{s}$ in the ball of radius $\delta$ about $\bm{0}$. The notation $\epsilon$ can be seen as the discretization parameter of the ball $B_{\delta}$ and we assume $\epsilon_1:=\epsilon /\delta$ is a fixed number.  We use $B^{\epsilon}_{\delta}$ to denote $B^{\epsilon}_{\delta}(\bm{0})$ for the rest of the paper.  An example of quadrature points in the horizon of an arbitrary point $\bm{x} \in \Omega$ is shown in \cref{fig:DiscreteBall}. 

\begin{figure}[H]
\centering
\scalebox{1}{\includegraphics{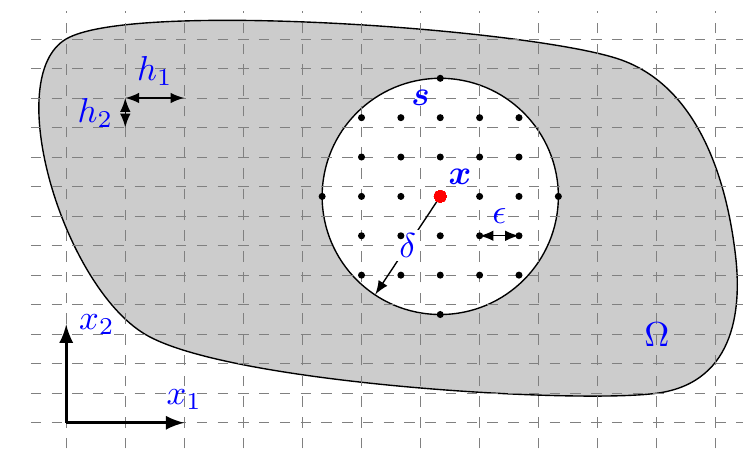}}
\caption{Quadrature points (black dots) are shown in the horizon of an arbitrary point $\bm{x}\in\Omega$. The dashed lines form the RK collocation grid which will be introduced in \cref{sec:RKCollocation}.
}
\label{fig:DiscreteBall}
\end{figure}
\noindent 
Due to the scaling of the nonlocal kernel $\rho_{\delta}(|\bm{s}|)$, see \cref{eqn:NonlocalKernelScaling}, we also have the scaling of the quadrature weights $\omega_{\delta}(|\bm{s}|)$ as
\begin{equation} \notag
\omega_{\delta}(\bm{s}) =\delta^{\nd} \omega \left( \frac{\bm{s}}{\delta} \right),
\end{equation}
where $\omega(\bm{s})$ is the quadrature weight  at $\bm{s} \in B^{\epsilon_1}_1$ and $\ep_1=\ep/\del$ is the discretization parameter of the unit ball.  As a consequence, we have a discrete version of \cref{eqn:BoundedMoment} as,
\begin{equation} \label{eqn:DBoundedMoment}
\sum\limits_{\bm{s} \in B^{\epsilon_1}_1}  \omega(\bm{s}) \rho(|\bm{s}|) |\bm{s}|^2 = \nd. 
\end{equation} 
However, unlike the nonlocal diffusion in \cite{Leng2019b}, \cref{eqn:DBoundedMoment} is not sufficient for the peridynamic Navier equation and we will revisit the construction of quadrature weights at the end of this section.

For $\bm{u}(\bm{x}) \in C^0(\mathbb{R}^\nd ; \, \mathbb{R}^{\nd})$, we can formulate the quasi-discrete counterpart of the nonlocal operators defined in the previous subsection.  The quasi-discrete bond-based peridynamic operator $\mathcal{L}^B_{\delta, \epsilon}$ is defined as
\begin{equation} \label{eqn:DiscreteNonlocalOperator}
\mathcal{L}^B_{\delta, \epsilon} \bm{u}(\bm{x}) = \sum\limits_{\bm{s} \in B^{\epsilon}_{\delta}}  \omega_{\delta}(\bm{s}) \rho_{\delta}(|\bm{s}|) \frac{\bm{s}}{|\bm{s}|} \otimes\frac{\bm{s}}{|\bm{s}|} (\bm{u}(\bm{x}+\bm{s})-\bm{u}(\bm{x})), \quad \forall \, \bm{x} \in \Omega. 
\end{equation} 
Similarly, the quasi-discrete nonlocal divergence operator $\mathcal{D}^{\epsilon}_{\delta}$ is formulated as , 
\begin{equation} \label{eqn:Ddivergence}  
\mathcal{D}^{\epsilon}_{\delta} \bm{u}(\bm{x}) = \sum\limits_{\bm{s} \in B^{\epsilon}_{\delta}}  \omega_{\delta}(\bm{s}) \rho_{\delta}(|\bm{s}|) \, \bm{s} \cdot (\bm{u}(\bm{x}+\bm{s}) - \bm{u}(\bm{x})), \quad \forall \, \bm{x} \in \Omega.
\end{equation}
A direct consequence of the quasi-discrete nonlocal divergence operator is the nonlocal dilatation,
\begin{equation} \label{eqn:Ddilatation} 
\theta^{\epsilon}(\bm{x}) = \frac{\nd}{m(\bm{x})} \mathcal{D}^{\epsilon}_{\delta} \bm{u}(\bm{x}), \quad \forall \, \bm{x} \in \Omega.
\end{equation}
We want to emphasize that $\theta^{\epsilon}$ is a continuous function with respect to $\bm{x}$ and its definition differs from \cref{eqn:Dilatation}.
We also have the quasi-discrete nonlocal gradient operator,
\begin{equation} \label{eqn:Dgradient} 
\mathcal{G}^{\epsilon}_{\delta} \theta^{\epsilon}(\bm{x}) = \sum\limits_{\bm{s} \in B^{\epsilon}_{\delta}} \omega_{\delta}(\bm{s})\rho_{\delta}(|\bm{s}|) \, \bm{s} \, (\theta^{\epsilon}(\bm{x}+\bm{s}) - \theta^{\epsilon}(\bm{x})), \quad \forall \, \bm{x} \in \Omega.
\end{equation}

Finally, we arrive at the linearized state-based quasi-discrete peridynamic Navier operator,
\begin{equation} \label{eqn:DPDoperator}
\mathcal{L}^S_{\delta, \epsilon} \bm{u}(\bm{x}) = \frac{C_{\alpha} \mu }{m(\bm{x})}\mathcal{L}^B_{\delta, \epsilon} \bm{u}(\bm{x}) + \frac{C_{\beta} \nd (\lambda - \mu) }{(m(\bm{x}))^2} \mathcal{G}^{\epsilon}_{\delta}  \mathcal{D}^{\epsilon}_{\delta} \bm{u}(\bm{x}), \quad \forall \, \bm{x} \in \Omega,
\end{equation}
and the static peridynamic Navier equation can be reformulated as
\begin{equation}  \label{eqn:DNonlocalEqn}
\begin{cases}
-\mathcal{L}^S_{\delta,\epsilon} \bm{u} = \bm{f}, & \textnormal{in } \Omega, \\
 \quad \quad \,\,\, \bm{u} = \bm{0}, & \textnormal{on } \Omega_{\mathcal{I}}.
\end{cases}
\end{equation}
Similar to \cref{eqn:Mixed}, if we let $\displaystyle p = (\lambda - \mu) \theta^{\epsilon}$, we can write \cref{eqn:DNonlocalEqn} as
\begin{equation}  \label{eqn:DMixed}
\begin{cases} 
\displaystyle -\frac{C_{\alpha} \, \mu }{m(\bm{x})}\mathcal{L}^B_{\delta, \epsilon} \bm{u}(\bm{x}) - \frac{C_{\beta}}{m(\bm{x})} \mathcal{G}^{\epsilon}_{\delta} p (\bm{x}) = \bm{f}(\bm{x}), & \bm{x} \in \Omega,  \\[8pt]
\displaystyle \frac{\nd \, ( \lambda - \mu)}{m(\bm{x})} \mathcal{D}^{\epsilon}_{\delta} \bm{u}(\bm{x}) - p (\bm{x})  = 0,  &\bm{x} \in \Omega, \\
\bm{u}(\bm{x}) = \bm{0},  & \bm{x} \in \Omega_{\mathcal{I}}.
\end{cases}
\end{equation}

If, for $\bm{u}$ being any quadratic polynomials,
\begin{equation} \label{eqn:DQuadraticExactness}
\mathcal{L}^S_{\delta, \epsilon} \bm{u} = \mathcal{L}^S_{\delta} \bm{u}, 
\end{equation}
the quasi-discrete peridynamic Navier operator $\mathcal{L}^S_{\delta, \epsilon}$ converges to $\mathcal{L}^S_{0}$ as $\delta$ goes to $0$ and $\epsilon$ is fixed.
\Cref{eqn:DQuadraticExactness} is not guaranteed by \cref{eqn:DBoundedMoment} but is fulfilled if the following is satisfied,  i.e., 
\begin{equation} \label{eqn:DBoundedMomentFourth2}
\sum\limits_{\bm{s} \in B^{\epsilon_1}_1}  \omega(|\bm{s}|) \rho(|\bm{s}|) \frac{s_i^2 s_j^2}{|\bm{s}|^2} = \int_{\bm{s} \in B_1}  \rho(|\bm{s}|) \frac{s_i^2 s_j^2}{|\bm{s}|^2} d\bm{s},
\end{equation} 
for $i, j = 1, \ldots, \nd$. It is easy to see that \cref{eqn:DBoundedMomentFourth2} is bounded because of \cref{eqn:BoundedMoment} and it is a reformulation of \cref{eqn:DWeightedVolume} by adding more constraints.

\section{RK collocation method} \label{sec:RKCollocation}
In this section, we discuss the RK collocation method and formulate the collocation equation on the peridynamic Navier equation and its quasi-discrete counterpart.
First, we introduce the collocation grid and let  $\square$ be a rectilinear Cartesian grid on $\mathbb{R}^{\nd}$, 
\begin{equation}  \notag
\square := \{ \bm{x_k} :=  \bm{k} \odot \bm{h} \mid \bm{k} \in \mathbb{Z}^{\nd} \},
\end{equation}
where  $\odot$ denotes component-wise multiplication, i.e.,
\begin{equation} \notag 
\bm{k} \odot \bm{h} = (k_1h_1, \ldots, k_{\nd}h_{\nd}),
\end{equation}
$ \bm{k}=(k_1, \ldots, k_\nd)$, and $\bm{h}=(h_1, \ldots, h_{\nd}) $ where $h_j$ is the discretization parameter in $j$-th dimension;
component-wise division is then denoted as $\oslash$: 
\begin{equation} \notag 
\bm{k} \oslash \bm{h} = \left(\frac{k_1}{h_1}, \ldots, \frac{k_{\nd}}{h_{\nd}} \right).
\end{equation}
We remark that the grid size $h_j$ can vary for different $j$ and we let $h_{\max} = \max^\nd_{j=1}h_j$ and $h_{\min} = \min^\nd_{j=1}h_j$. For instance, in two dimension, rectangular grids are allowed. In addition, the grid $\square$ is quasi-uniform such that $\bm{h}$ can be rewritten as
\begin{equation} \label{eqn:GridVector} 
\bm{h} =  h_{\max} \bm{\hat{h}}\,,
\end{equation}
where $\bm{\hat{h}}$ is a fixed vector with the maximum component being 1 and the minimum component being bounded below. 

Next, we let $S(\square)$ be the trial space equipped with RK basis on $\square$, i.e., $S(\square) =\textnormal{span}\{\Psi_{\bm{k}}(\bm{x}) \mid \bm{k} \in \mathbb{Z}^{\nd}\} $. 
The RK basis function $\Psi_{\bm{k}}(\bm{x})$ is given as 
\begin{equation}   \label{eqn:RKShape}
\Psi_{\bm{k}}(\bm{x}) =\prod^{\nd}_{j=1} \phi \left( \frac{|x_j-x_{k_j}|}{2h_j} \right) ,
\end{equation}
where $x_{k_j} = k_jh_j $ is the $j$-th component of $\bm{x_k}$, $2h_j$ is the RK support in the $j$-th dimension, and $\phi(x)$ is the cubic B-spline function
\begin{equation}  \label{eqn:CubicSpline}
\phi(x) = 
\begin{cases}
\frac{2}{3}-4x^2+4x^3 , \quad & 0 \leq x \leq \frac{1}{2}, \\
\frac{4}{3}(1-x)^3, & \frac{1}{2} \leq x \leq1, \\
0, & \textnormal{otherwise}.
\end{cases}
\end{equation}

\begin{remark}
For the simplicity of presentation, we choose the RK support size $\bm{a}=a_0\bm{h}$ where $a_0=2$ in the paper but the analysis works for general even number $a_0$ \cite{Leng2019a,Leng2019b}.
\end{remark}

Thus, the RK basis function can reproduce linear polynomials \cite{Leng2019a,Li1996,Liu1995}, i.e., 
\begin{equation}  \label{eqn:LinearRC}
\sum_{\bm{k} \in \mathbb{Z}^{\nd}} \Psi_{\bm{k}}(\bm{x}) \bm{x_k}^{\bm{\alpha}} = \bm{x}^{\bm{\alpha}},	\textnormal{ for } |\bm{\alpha}|=1.
\end{equation}
For $u \in  C^0(\mathbb{R})$, we define the restriction to $\square$ by
\begin{equation} \notag
r^hu : = (u(\bm{x_k}))_{\bm{k}\in \mathbb{Z}^{\nd}}, 
\end{equation}
and the restriction to ($\square \cap \Omega$) as 
\[
r^h_{\Omega} u  := (u(\bm{x_k})), \quad \bm{x}_{\bm{k}} \in (\square \cap \Omega).
\]
For a sequence $(u_{\bm{k}})_{\bm{k}\in\mathbb{Z}^{\nd}}$ on $\mathbb{R}$, the RK interpolant operator is defined by  
\begin{equation} \notag
i^h(u_{\bm{k}}) : = \sum_{\bm{k} \in \mathbb{Z}^{\nd}}  \Psi_{\bm{k}}(\bm{x}) u_{\bm{k}}.
\end{equation}
For $j = 1, \ldots , \nd$, we denote $u_j(\bm{x}): \mathbb{R}^{\nd} \to \mathbb{R}$ the $j$-th component of a vector field $\bm{u}(\bm{x})=[u_1(\bm{x}),\ldots, u_\nd(\bm{x})]^T$ and denote $(u_{j,\bm{k}})$ the $j$-th component of the vector sequence  
\[
(\bm{u_{k}}) = [(u_{1,\bm{k}})_{\bm{k} \in \Z^\nd}, \ldots, (u_{\nd,\bm{k}})_{\bm{k} \in \Z^\nd}]^T .
\]
Then we let
\begin{equation} \notag 
\Pi^h := i^hr^h
\end{equation}
be the interpolation projector mapping from $C^0(\mathbb{R}^{ \nd })$ to  $S(\square)$. Therefore, we can write
\begin{equation} \notag
\Pi^h\bm{u} : = [\Pi^hu_1, \ldots ,\Pi^hu_{\nd}]^T,
\end{equation}
where $\Pi^h u_j(\bm{x})$ is the RK approximation of $u_j(\bm{x})$,
\begin{equation} \notag 
\Pi^h u_j(\bm{x}) = \sum\limits_{\bm{k} \in \mathbb{Z}^{\nd}} \Psi_{\bm{k}}(\bm{x})u_j(\bm{x_k}).
\end{equation}

Finally, we apply RK approximation on both $\bm{u}$ and $\theta$, back-substitute $\theta$ into the first equation of \cref{eqn:Mixed} and obtain 
\[
\mathcal{L}^S_{\delta}\Pi^h\bm{u} = \frac{C_{\alpha} \, \mu }{m(\bm{x})}\mathcal{L}^B_{\delta} \Pi^h \bm{u}  + \frac{C_{\beta} \nd \, ( \lambda - \mu)}{(m(\bm{x}))^2} \mathcal{G}_{\delta} \Pi^h (\mathcal{D}_{\delta} \Pi^h \bm{u}).
\]
Following a similar procedure, we arrive at
\[
\mathcal{L}^S_{\delta, \epsilon}\Pi^h\bm{u} = \frac{C_{\alpha} \, \mu }{m(\bm{x})}\mathcal{L}^B_{\delta, \epsilon} \Pi^h \bm{u}  + \frac{C_{\beta} \nd \, ( \lambda - \mu)}{(m(\bm{x}))^2} \mathcal{G}^{\epsilon}_{\delta} \Pi^h (\mathcal{D}^{\epsilon}_{\delta} \Pi^h \bm{u}).
\]
Therefore the RK collocation scheme of \cref{eqn:NonlocalEqn,eqn:DNonlocalEqn} can be written in the following forms.  
Find a function $\bm{u} \in S(\square \cap \Omega; \mathbb{R}^{\nd} )$, such that
\begin{equation}  \label{eqn:CollocationScheme}
-r_{\Omega}^h \mathcal{L}^S_{\delta} \bm{u} = r^h_{\Omega} \bm{f}, 
\end{equation}
and 
\begin{equation}  \label{eqn:DCollocationScheme}
-r_{\Omega}^h \mathcal{L}^S_{\delta, \, \epsilon} \bm{u} = r^h_{\Omega} \bm{f},
\end{equation}
where for $\bm{u}$ coming from the trial space, we have abused the notations and let $ \mathcal{L}^S_{\delta}\bm{u} $ and $\mathcal{L}^S_{\delta, \epsilon}\bm{u} $ represent, 
\[
\mathcal{L}^S_{\delta}\bm{u} = \frac{C_{\alpha} \, \mu }{m(\bm{x})}\mathcal{L}^B_{\delta} \bm{u}  + \frac{C_{\beta} \nd \, ( \lambda - \mu)}{(m(\bm{x}))^2} \mathcal{G}_{\delta} \Pi^h (\mathcal{D}_{\delta} \bm{u}),
\]
and 
\[
\mathcal{L}^S_{\delta, \epsilon}\bm{u} = \frac{C_{\alpha} \, \mu }{m(\bm{x})}\mathcal{L}^B_{\delta, \epsilon}  \bm{u}  + \frac{C_{\beta} \nd \, ( \lambda - \mu)}{(m(\bm{x}))^2} \mathcal{G}^{\epsilon}_{\delta} \Pi^h (\mathcal{D}^{\epsilon}_{\delta}  \bm{u}) .
\]
The main contribution of this paper is to show the convergence analysis of the two collocation schemes.

\section{Convergence analysis of RK collocation method} \label{sec:Convergence}
In this section, we show the convergence analysis of the RK collocation scheme \cref{eqn:CollocationScheme},
which is used in \cite{Pasetto2018} without any analysis. A convergence proof for nonlocal diffusion problems is provided in \cite{Leng2019b}, and the analysis is extended to the peridynamic Navier equation in this work.
The main objective is to show that the solution of the numerical scheme converges to the nonlocal problem for a fixed $\delta$ and $h_{\max}$ vanishes,
and to the correct local problem as $\delta$ and grid size $h_{\max}$ both go to zero. 

\subsection{Stability analysis of the RK collocation method} \label{subsec:stability}
In this subsection, we show the stability analysis of the RK collocation method. We first define a norm in the space of vector-valued sequences by 
\begin{equation} \label{eqn:l2norm}
|(\bm{u}_{\bm{k}})_{\bm{k} \in {\mathbb{Z}^{\nd}}}|_h := \| i^h(\bm{u}_{\bm{k}})  \|_{L^2(\mathbb{R}^{\nd}; \, \mathbb{R}^{\nd})} \,\, .
\end{equation}
For a sequence $(\bm{u}_{\bm{k}})$ only defined for $\bm{k}$ being in a subset of $\mathbb{Z}^{\nd}$, we can always extend $(\bm{u}_{\bm{k}})$ by zero to $\bm{k} \in \mathbb{Z}^{\nd}$. 
Then without further explanation,
$|(\bm{u}_{\bm{k}}) |_h $  is always understood as \eqref{eqn:l2norm} 
with zero extension. We next borrow the idea from \cite{Costabel1992, Leng2019b} and compare the RK collocation scheme with the Galerkin scheme using Fourier analysis.

\begin{thm} \label{thm:stability}
For any $\delta \in(0, \delta_0]$, there exists a constant $C$ that depends on $\Omega$ and $\delta_0$, such that for $\bm{u} \in S(\square \cap \Omega; \mathbb{R}^{\nd})$, 
\[
|r^h_{\Omega}(-\mathcal{L}^S_{\delta} \bm{u})|_{h} \geq C \| \bm{u}\| _{L^2(\mathbb{R}^{\nd}; \mathbb{R}^{\nd})}.
\]
\end{thm}
We need some intermediate results before proving \cref{thm:stability}. We define a scalar product in $l^2(\mathbb{Z}^{\nd}; \mathbb{C}^{\nd})$,
\begin{equation} \notag
\begin{aligned}
((\bm{u_{k}}), (\bm{v_k}))_{l^2} &:= \sum_{\bm{k} \in \mathbb{Z}^{\nd}}u_{1,\bm{k}}\overline{v_{1,\bm{k}}}+ \ldots
+\sum_{\bm{k} \in \mathbb{Z}^{\nd}}u_{\nd,\bm{k}}\overline{v_{\nd,\bm{k}}} \\
&= \sum_{j=1}^{\nd}\sum_{\bm{k} \in \mathbb{Z}^{\nd}}u_{j,\bm{k}}\overline{v_{j,\bm{k}}} .
\end{aligned}
\end{equation}
The Fourier series of a vector-valued sequence $(\bm{u_k})$ is defined as 
\[
\bm{\widetilde{u}}(\bm{\xi}) = [\tilde{u}_1(\bm{\xi}), \ldots, \tilde{u}_{\nd}(\bm{\xi}) ]^T
\]
and the $j$-th component of $\bm{\widetilde{u}}(\bm{\xi})$ is  
\begin{equation} \notag
\tilde{u}_j(\bm{\xi}) := \sum_{\bm{k} \in \mathbb{Z}^{\nd}} e^{-i\bm{k} \cdot \bm{\xi}}u_{j, \bm{k}},
\end{equation}
where 
\begin{equation} \notag
u_{j,\bm{k}} = (2\pi)^{-\nd} \int_{\bm{Q}} e^{i\bm{k} \cdot \bm{\xi}}\tilde{u}_j(\bm{\xi})d{\bm{\xi}},
\end{equation}
for $\bm{Q}:=(-\pi, \pi)^{\nd}$.

We present the Fourier symbol of the peridynamic Navier operator $-\mathcal{L}^S_{\delta}$ in the next lemma. The Fourier transform of $\bm{u} \in S_{\delta}$ is defined by
\begin{equation} \notag 
\bm{\widehat{u}}(\bm{\xi}) := \int_{\mathbb{R}^{\nd}} e^{-i \bm{x} \cdot \bm{\xi}} \bm{u}(\bm{x}) d\bm{x}.
\end{equation} 
The proof of the following lemma can be found in \cref{pf:lemFPDsymbol}. 
\begin{lemma} \label{lem:FPDsymbol}
The Fourier symbol of the peridynamic Navier operator $\mathcal{L}^S_{\delta}$ is given by
\begin{equation} \label{eqn:FPDNoperator}
-\widehat{\mathcal{L}^S_{\delta} \bm{u}}(\bm{\xi}) = \bm{M}^S_{\delta}(\bm{\xi}) \widehat{\bm{u}} (\bm{\xi}),
\end{equation}
where the Fourier symbol $\bm{M}^S_{\delta}(\bm{\xi})$ is a $\nd \times \nd$ matrix and consists of two parts,
\begin{equation} \label{eqn:PDNsymbol}
\bm{M}^S_{\delta}(\bm{\xi}) = \bm{M}^B_{\delta}(\bm{\xi}) + \bm{M}^D_{\delta}(\bm{\xi}),
\end{equation}
where 
\begin{equation} 
\bm{M}^B_{\delta}(\bm{\xi}) = \frac{C_{\mu}}{\delta^2}  p_{1}(\delta|\bm{\xi}|)\left( \bm{I}_{\nd} - \vec{\bm{\xi}} \vec{\bm{\xi}}^{\,\,T} \right) 
+ \frac{C_{\mu}}{\delta^2}   q_{1}(\delta|\bm{\xi}|)   \vec{\bm{\xi}} \vec{\bm{\xi}}^{\,\,T} ,
\end{equation}
and 
\begin{equation} 
\bm{M}^D_{\delta}(\bm{\xi}) = \frac{C_{\lambda,\mu}}{\delta^2}  \left(b_{1}(\delta|\bm{\xi}|) \right)^2 \vec{\bm{\xi}} \vec{\bm{\xi}}^{\,\,T} ,
\end{equation}
where $\bm{I}_{\nd}$ is the $\nd$-dimensional identity matrix, \( \displaystyle \vec{\bm{\xi}} =\frac{\bm{\xi}}{|\bm{\xi}|} \) is the unit vector in the direction of $\bm{\xi}$, \( C_{\mu} = C_{\alpha}\mu/\nd \) and  \(  C_{\lambda, \mu} =C_{\beta}(\lambda- \mu) \) are material dependent constants, the scalars $p_1( |\bm{\xi}| ), q_1(|\bm{\xi}| )$ and $b_1( |\bm{\xi}| )$ are given by
\begin{equation} \label{eqn:scalarP1}
 p_1(|\bm{\xi}| ) = \int_{B_{1}} \rho( |\bm{s}|) \frac{s^2_1}{|\bm{s}|^2}(1-\textnormal{cos}(|\bm{\xi}|s_{\nd})) d\bm{s},  
\end{equation}
\begin{equation} \label{eqn:scalarQ1}
q_1(|\bm{\xi}| ) = \int_{B_{1}} \rho( |\bm{s}|) \frac{s^2_{\nd}}{|\bm{s}|^2}(1-\textnormal{cos}(|\bm{\xi}|s_{\nd})) d\bm{s}, 
\end{equation}
\begin{equation} \label{eqn:scalarB1} 
b_1(\bm{|\xi|}) = \int_{B_{1}} \rho(|\bm{s}|) s_{\nd} \, \textnormal{sin}( |\bm{\xi}|s_{\nd}) d\bm{s}.
\end{equation}
\end{lemma}

From \cref{eqn:PDNsymbol}, if $C_{\lambda, \mu} \geq 0$ we can immediately see that the Fourier symbol $\bm{M}^S_{\delta}(\bm{\xi})$ is positive definite.
\begin{lemma} \label{lem:PositiveDefinite}
Assume $\lambda \geq \mu $, the Fourier symbol $\bm{M}^S_{\delta}(\bm{\xi})$ is positive definite for any $\bm{\xi} \neq \bm{0}$. 
\end{lemma}
\begin{proof}
By observation, $\bm{M}^S_{\delta} (\bm{\xi}) $ is a real matrix. Moreover, from \cref{eqn:scalarP1,eqn:scalarQ1} we know that 
\[ p_{1}(\delta|\bm{\xi}|), q_{1}(\delta|\bm{\xi}|) > 0, \quad \textnormal{ for } \delta |\bm{\xi}| \neq 0. \]
Without loss of generality, we let $\bm{v}$ be a unit vector so  $|\bm{v}^T \vec{\bm{\xi}} \, | \leq |\bm{v}|$ because $|\vec{\bm{\xi}} \, |=1$. Then, we have
\begin{align*}
\delta^2 \bm{v}^T \bm{M}^S_{\delta}(\bm{\xi}) \bm{v} &\geq C_{\mu} p_{1}(\delta|\bm{\xi}|)\bm{v}^T \left(\bm{I}_{\nd}-\vec{\bm{\xi}}\vec{\bm{\xi}}^{\,\, T} \right) \bm{v} \\
& \quad + \left[ C_{\mu} q_{1}(\delta|\bm{\xi}|) + C_{\lambda, \mu} \left(b_{1}(\delta|\bm{\xi}|) \right)^2 \right]\bm{v}^T \vec{\bm{\xi}} \vec{\bm{\xi}}^{\,\, T} \bm{v}, \\
&= C_{\mu} p_{1}(\delta|\bm{\xi}|) \left(|\bm{v}|^2 - |\bm{v}^T \vec{\bm{\xi}} \, |^2 \right) \\
&\quad + \left[ C_{\mu} q_{1}(\delta|\bm{\xi}|)
+ C_{\lambda, \mu} \left(b_{1}(\delta|\bm{\xi}|) \right)^2 \right] |\bm{v}^T \vec{\bm{\xi}} \, |^2, \\
& > 0,
\end{align*}
where we have used the assumption that $C_{\lambda, \mu} = \lambda - \mu \geq 0.$
\end{proof}

\begin{remark}
In order to show the positive definiteness of $\bm{M}^S_{\delta}$ for more general $\lambda$ and $\mu$, we need more details on the nonlocal kernel ($\rho(|\bm{s}|)$) which is beyond the scope of this paper and we assume $\lambda \geq \mu$ to avoid such discussion. For materials that satisfy such constraint, their Poisson ratio $\nu$ have to be in $ [0.25, 0.5)$. However, the well-posedness of \cref{eqn:NonlocalEqn} proved in \cite{mengesha2014nonlocal} infers that $\bm{M}^S_{\delta}$ is positive definite even without this assumption.
\end{remark}

The peridynamic Navier operator $\mathcal{L}^S_{\delta}$ defines two discrete sesquilinear forms:
\begin{equation} \label{eqn:IGalerkin}
\left(i^h(\bm{u_{k}}), -\mathcal{L}^S_{\delta}i^h(\bm{v_{k}})\right) = \sum_{j, \, j'=1}^{\nd} \sum_{\bm{k}, \, \bm{k'} \in \mathbb{Z}^{\nd}} u_{j,\, \bm{k}} \left((\Psi_{\bm{k}}), -\mathcal{L}^S_{\delta}(\Psi_{\bm{k'}}) \right)\overline{v_{j',\,  \bm{k'}}} \, ,
\end{equation}
and
\begin{equation} \label{eqn:ICollocation}
\left((\bm{u_{k}}), -r^h\mathcal{L}^S_{\delta}i^h(\bm{v_{k}})\right)_{l^2} =\prod_{j=1}^{\nd}h_j \sum_{j, \, j'=1}^{\nd} \sum_{ \bm{k}, \bm{k'} \in \mathbb{Z}^{\nd}}  u_{j,\, \bm{k}}  \left(-\mathcal{L}^S_{\delta}(\Psi_{\bm{k'}})\right)(\bm{x_{k}})\overline{v_{j', \, \bm{k'}}} \, .
\end{equation}
\Cref{eqn:IGalerkin} defines a quadratic form corresponding to the Galerkin method, meanwhile, \cref{eqn:ICollocation} corresponds to the collocation method. The two quadratic forms \cref{eqn:IGalerkin,eqn:ICollocation} are compared as follows. The proof is similar to \cite[Lemma 4.2]{Leng2019b} so we provide it in \cref{pf:lemGCF}. 

\begin{lemma} \label{lem:GCF}
Let $\bm{\widetilde{u}}(\bm{\xi})$ and $\bm{\widetilde{v}}(\bm{\xi})$ be the Fourier series of the sequences $(\bm{u_{k}}), (\bm{v_{k}}) \in l^1(\mathbb{Z}^{\nd};\mathbb{C}^{\nd})$ respectively and  the RK interpolation of two sequences are expressed as \(i^h(\bm{u}_{\bm{k}}) \allowbreak = [i^h(u_{1,\bm{k}}), \ldots, i^h(u_{\nd,\bm{k}})]^T \) and \( \displaystyle i^h(\bm{v}_{\bm{k}}) = [ i^h(v_{1,\bm{k}}), \allowbreak \ldots, i^h(v_{\nd,\bm{k}}) ]^T \). 
Then
\begin{enumerate}[label=(\roman*)]
\item \(\displaystyle \left(i^h(\bm{u_{k}}), -\mathcal{L}^S_{\delta} i^h(\bm{v_{k}})\right) = \left(2\pi \right)^{-\nd}  \mathlarger{\int}_{\bm{Q}} \bm{\widetilde{u}}(\bm{\xi}) \cdot \bm{M}_G(\delta, \bm{h}, \bm{\xi}) \overline{\bm{\widetilde{v}}(\bm{\xi})}  d\bm{\xi} \), \\[.2pt] \label{FGalerkin}
\item \(\displaystyle \left((\bm{u_k}), -r^h\mathcal{L}^S_{\delta}i^h(\bm{v_{k}})\right)_{l^2} = \left(2\pi \right)^{-\nd}  \mathlarger{\int}_{\bm{Q}} \bm{\widetilde{u}}(\bm{\xi}) \cdot \bm{M}_C(\delta, \bm{h}, \bm{\xi})  \overline{\widetilde{\bm{v}}(\bm{\xi})} d\bm{\xi} \), \\[.1pt] \label{FCollocation}
\item There exists a constant $C > 0$ independent of $\delta, \bm{h}$ and $\bm{\xi}$ such that $\bm{M}_C(\delta, \bm{h}, \bm{\xi}) - C \bm{M}_G(\delta, \bm{h}, \bm{\xi})$ is positive definite for any $\bm{\xi} \neq \bm{0}$, \label{FGCEquivalent}
\end{enumerate}
where $\bm{M}_G$ and $\bm{M}_C$ are defined as
\begin{equation} \label{eqn:lambdaG}
\begin{aligned}
\bm{M}_G(\delta, \bm{h}, \bm{\xi}) & =  2^{8\nd}\sum_{\bm{r} \in \mathbb{Z}^{\nd}} \bm{M}^B_{\delta} \left( (\bm{\xi} + 2 \pi \bm{r})\oslash\bm{h} \right) \prod_{j=1}^{\nd}h_j\left(\frac{\sin(\xi_j/2)}{\xi_j+  2\pi r_j}\right)^8    \\ 
& \quad + 2^{8\nd+4}\sum_{\bm{r} \in \mathbb{Z}^{\nd}} \bm{M}^D_{\delta} \left( (\bm{\xi} + 2 \pi \bm{r})\oslash\bm{h} \right) \prod_{j=1}^{\nd}h_j\left(\frac{\sin(\xi_j/2)}{\xi_j+  2\pi r_j}\right)^{12}  , 
\end{aligned}
\end{equation}
\begin{equation} \label{eqn:lambdaC}
\begin{aligned}
\bm{M}_C(\delta, \bm{h}, \bm{\xi}) & =  2^{4\nd}\sum_{\bm{r} \in \mathbb{Z}^{\nd}} \bm{M}^B_{\delta}\left((\bm{\xi} + 2 \pi \bm{r}) \oslash \bm{h} \right)  \prod_{j=1}^{\nd}h_j\left(\frac{\sin(\xi_j/2)}{\xi_j+ 2 \pi r_j}\right)^4 \\
& \quad + 2^{4\nd+4}\sum_{\bm{r} \in \mathbb{Z}^{\nd}} \bm{M}^D_{\delta}\left((\bm{\xi} + 2 \pi \bm{r}) \oslash \bm{h} \right)  \prod_{j=1}^{\nd}h_j\left(\frac{\sin(\xi_j/2)}{\xi_j+ 2 \pi r_j}\right)^8.
\end{aligned}
\end{equation}
\end{lemma}

Finally, we are ready to prove \cref{thm:stability} using \cref{lem:GCF}.
\begin{proof}[Proof of \cref{thm:stability}]

For $\bm{u}=i^h(\bm{u_k})  \in S(\square \cap \Omega; \, \mathbb{R}^{\nd} )$, we have 
\begin{align*}
| (\bm{u}_{\bm{k}})|_h \cdot |r^h_{\Omega} (-\mathcal{L}^S_{\delta}  \bm{u})|_h & \geq C |( (\bm{u}_{\bm{k}}),  r^h_{\Omega} (-\mathcal{L}^S_{\delta} \bm{u}))_{l^2} |,  \\
   &= C |((\bm{u}_{\bm{k}}),  r^h (-\mathcal{L}^S_{\delta} i^h(\bm{u_k})))_{l^2} |,  \\
   &\geq C |(i^h(\bm{u_k}), (-\mathcal{L}^S_{\delta} i^h(\bm{u_k}))) |,  \\
   &\geq C \| \bm{u}\|^2_{L^2(\mathbb{R}^{\nd}; \, \mathbb{R}^{\nd})}. 
\end{align*}
The first line comes from the Cauchy-Schwartz inequality and the third line is an adaption of \cite[lemma 4.4]{Leng2019b}. 
\end{proof}

\subsection{Consistency analysis of the RK collocation}
 In this subsection, we first show the consistency of the RK collocation method and then present the convergence result using the stability (\cref{subsec:stability}) and consistency analysis. The RK collocation scheme converges to the nonlocal solution as grid size $h_{\max}$ goes to zero with a fixed $\delta$ and to the corresponding local limit as $\delta$ and $h_{\max}$ both vanish. The major ingredient for proving the asymptotic compatibility is the synchronized convergence property of the RK approximation and it is a well established result when the RK support size is selected carefully. Therefore we skip the proof and present the result in the following lemma, and we refer the readers to \cite{chen2017meshfree,Leng2019b, Li1996, Li1998synchronized} for more details. For the rest of the paper, we adopt the following notations for a vector-valued function 
$\bm{u}\in C^n(\R^\nd; \R^\nd)$,
\begin{equation} \notag 
\begin{aligned}
|\bm{u}|_\infty &= \sup_{1 \leq j \leq \nd} \sup_{\bm{x}\in \R^\nd} |u_j(\bm{x})|,  \textnormal{ and }\\
|\bm{u}^{(l)}|_{\infty} & = \sup_{1 \leq j \leq \nd} \sup_{|\bm{\beta}|=l}\sup_{\bm{y} \in \mathbb{R}^{\nd}}|D^{\bm{\beta}} u_j(\bm{y})|, \quad 1\leq l \leq n. 
\end{aligned}
\end{equation}

\begin{lemma} \label{lem:synchronizedconvergence}
\textbf{(Synchronized Convergence)}
Assume a scalar valued function $u \in C^4(\mathbb{R}^{\nd}) $ and  $\Pi^h u$  is the RK interpolation with
the shape function given by \cref{eqn:RKShape}. $\Pi^h u$ has synchronized convergence, namely 
\begin{equation} \notag 
\left|D^{\bm{\alpha}}(\Pi^h u - u )  \right|_\infty \leq C |u^{(|\bm{\alpha}| + 2)} |_{\infty} h_{\textnormal{max}}^2, \quad \textnormal{for } |\bm{\alpha}| = 0, 1, 2,
\end{equation}
where $C$ is a generic constant independent of $h_{\textnormal{max}}$\,.
\end{lemma}

Next, we study the truncation error of the RK collocation method on the peridynamic Navier operator. 

\begin{lemma} \label{lem:consistency}
\textbf{(Uniform consistency)} Assume $\bm{u} \in C^4(\mathbb{R}^{\nd};\mathbb{R}^{\nd})$, then
\begin{equation} \notag 
| r^h\mathcal{L}^S_{\delta} \Pi^h{\bm{u}} - r^h \mathcal{L}^S_{\delta} \bm{u} |_{h} \leq C h_{\max}^2 |\bm{u}^{(4)}|_{\infty},
\end{equation}
where $C$ is independent of $h_{\max}$ and $\delta$.  
\end{lemma}

\begin{proof}
We first define the interpolation error of $u_j(\bm{x})$, for $\bm{x} \in \mathbb{R}^{\nd}$, and $j=1, \ldots, \nd$, as
\begin{equation} \notag 
E_j(\bm{x}) = \Pi^h u_j(\bm{x}) - u_j(\bm{x}), 
\end{equation}
then

\begin{equation}
\bm{E}(\bm{x}) = [E_1(\bm{x}), \ldots, E_{\nd}(\bm{x})]^T.
\end{equation}
By restricting on the the grid point $\bm{x_k}$, for $i=1, \ldots, \nd$, the truncation error of $\mathcal{L}^B_{\delta}$ is given as
\begin{equation} \label{eqn:bondTruncationError} 
\begin{aligned}
\left| \left[ \mathcal{L}^B_{\delta} \left(\Pi^h \bm{u} - \bm{u} \right) \right]_i (\bm{x_k})  \right| &=  \left| \left[ \mathcal{L}^B_{\delta}\bm{E} \right]_i (\bm{x_k}) \right|, \\
&=\left| \sum^{\nd}_{j=1} \mathlarger{\int}_{B_{\delta}} \rho_{\delta}(|\bm{s}|) \frac{s_i s_j}{|\bm{s}|^2} \left( E_j(\bm{x_k}+\bm{s}) - E_j(\bm{x_k}) \right) d\bm{s} \right|. 
\end{aligned}
\end{equation}
Next, using \cref{lem:synchronizedconvergence}, we can bound the interpolation error as
\begin{equation} \label{eqn:RemainderError}
\begin{aligned}
| E_j(\bm{x_k}+\bm{s}) + E_j(\bm{x_k}-\bm{s}) - 2E_j(\bm{x_k})| & \leq  C |\bm{s}|^2 \max_{|\bm{\alpha}|=2}\left|D^{\bm{\alpha}}E_j(\bm{x}) \right|_{\infty},  \\
& \leq  C |\bm{s}|^2 |u_j^{(4)}|_{\infty} h_{\max}^2 \, .
\end{aligned}
\end{equation}
Combining \cref{eqn:bondTruncationError,eqn:RemainderError}, we have
\begin{equation} \label{eqn:BondBasedTruncationError}
\begin{aligned}
\left| [ \mathcal{L}^B_{\delta} \left(\Pi^h \bm{u} - \bm{u}\right) (\bm{x_k}) ]_i \right| & \leq C  h_{\max}^2 \sum^{\nd}_{j=1}|u_j^{(4)}|_{\infty}  \int_{B_{\delta}} \rho_{\delta}(|\bm{s}|)|s_i| |s_j| \, d\bm{s} , \\
 & \leq C h_{\max}^2 \left| \bm{u}^{(4)} \right|_{\infty} ,
\end{aligned}
\end{equation}
where we have used \cref{eqn:BoundedMoment} and $C > 0$ is a generic constant depending the dimension, $\nd$. 

Next, we define the interpolation error of the nonlocal dilatation as
\begin{equation} \label{eqn:DilationTruncationError}
\begin{aligned}
E_{\theta} &= \Pi^h \mathcal{D}_{\delta} \Pi^h \bm{u} - \mathcal{D}_{\delta} \bm{u}, \\
&=\Pi^h (\mathcal{D}_{\delta} \Pi^h \bm{u} - \mathcal{D}_{\delta}\bm{u}) + \Pi^h \mathcal{D}_{\delta}\bm{u} -\mathcal{D}_{\delta} \bm{u}, \\
&=  \Pi^h \mathcal{D}_{\delta} \bm{E} +  (\Pi^h \theta - \theta) ,
\end{aligned}
\end{equation}
where we have used the definition of the nonlocal dilatation \cref{eqn:Dilatation}. There are two RK interpolation projectors ($\Pi^h$) in the first line of \cref{eqn:DilationTruncationError} because we apply RK interpolation to $\bm{u}$ and $\theta$, then back-substitute $\theta$ to get a pure displacement form. 
The nonlocal gradient operator acting on $E_{\theta}$ can be bounded by
\begin{equation} \label{eqn:GradientTruncationError}
\begin{aligned}
\left| [ \mathcal{G}_{\delta} E_{\theta} (\bm{x_k}) ]_i \right| & =  \left| \int_{B_{\delta}} \rho_{\delta}(|\bm{t}|) t_i \left( E_{\theta}(\bm{x_k}+\bm{t}) - E_{\theta}(\bm{x_k}) \right) d\bm{t} \right|, \\
 &\leq \max_{|\bm{\beta}|=1}\left|D^{\bm{\beta}}E_{\theta} \right|_{\infty} \int_{B_{\delta}} \rho_{\delta}(|\bm{t}|)|t_i| |\bm{t}|  \, d\bm{t} \, , \\
 & \leq C \max_{|\bm{\beta}|=1}\left|D^{\bm{\beta}}E_{\theta} \right|_{\infty}, \\
 & \leq C \max_{|\bm{\beta}|=1}\left|D^{\bm{\beta}} \Pi^h \mathcal{D}_{\delta} \bm{E}  \right|_{\infty} + C \max_{|\bm{\beta}|=1}\left|D^{\bm{\beta}} (\Pi^h \theta - \theta)  \right|_{\infty} , 
\end{aligned}
\end{equation}
for $i =1, \ldots, \nd$.
We can bound the first term in the last line of \cref{eqn:GradientTruncationError} by
\begin{equation} \label{eqn:FirstOrderRemainderError}
\begin{aligned}
\max_{|\bm{\beta}|=1}\left| D^{\bm{\beta}}  \Pi^h \mathcal{D}_{\delta} \bm{E} \right|_{\infty} & =  \max_{|\bm{\beta}|=1}\left| \sum^{\nd}_{j=1}  D^{\bm{\beta}} \Pi^h \int_{B_{\delta}} \rho_{\delta}(|\bm{s}|) s_j \left( E_j(\bm{x}+\bm{s}) - E_j(\bm{x}) \right) d\bm{s} \right|_{\infty}, \\
& \leq  \sum^{\nd}_{j=1} \int_{B_{\delta}} \rho_{\delta}(|\bm{s}|) |s_j| |\bm{s}|  d\bm{s} \max_{|\bm{\alpha}|=|\bm{\beta}|=1}  \left| D^{\bm{\beta}} \Pi^h D^{\bm{\alpha}}E_j(\bm{x}) \right|_{\infty},  \\
& \leq C h^2_{\max} \left|\bm{u}^{(4)} \right|_{\infty},
\end{aligned}
\end{equation}
where the derivation of the second line to the last can be obtained by similar expansion of \cite[eq.(36)]{Li1998synchronized} and the results of \cite[Lemma 4.1]{Leng2019a}, and we have used \cref{lem:synchronizedconvergence}. 
Next, we have the bound of the second term in the last line of \cref{eqn:GradientTruncationError} as
\begin{equation} \label{eqn:ThetaApproximationError}
\max_{|\bm{\beta}|=1}\left| D^{\bm{\beta}} (\Pi^h \theta - \theta) \right|_{\infty} \leq  C \left|\theta^{(3)}\right|_{\infty}h^2_{\max} \, ,
\end{equation}
and $\left|\theta^{(3)}\right|_{\infty}$ is bounded by
\begin{equation} \label{eqn:ThetaThirdError}
\begin{aligned}
\left|\theta^{(3)}\right|_{\infty} &= \frac{{\nd}}{m(\bm{x})}\max_{|\bm{\beta}|={3}}  \left| \sum^{{\nd}}_{j=1}  D^{\bm{\beta}}   \int_{B_{\delta}} \rho_{\delta}(|\bm{s}|) s_j  \left( u_j(\bm{x}+\bm{s}) - u_j(\bm{x}) \right) d\bm{s} \right|_{\infty}, \\
& \leq  C \sum^{{\nd}}_{j=1} \int_{B_{\delta}} \rho_{\delta}(|\bm{s}|) |s_j| |\bm{s}|  d\bm{s} \max_{|\bm{\alpha}|=4}  \left|D^{\bm{\alpha}}  u_j(\bm{x}) \right|_{\infty}, \\
& \leq C \left|\bm{u}^{(4)} \right|_{\infty} .
\end{aligned}
\end{equation}

\noindent 
By collecting \cref{eqn:GradientTruncationError,eqn:FirstOrderRemainderError,eqn:ThetaApproximationError,eqn:ThetaThirdError}, the truncation error of the composition of the nonlocal gradient and divergence operators is bounded by
\begin{equation}  \label{eqn:GDcompositionerror}
\left| [ (\mathcal{G}_{\delta} \Pi^h\mathcal{D}_{\delta}\Pi^h \bm{u} - \mathcal{G}_{\delta}\mathcal{D}_{\delta} \bm{u}) (\bm{x_k}) ]_i \right| \leq C h^2_{\max} \left|\bm{u}^{(4)} \right|_{\infty} .
\end{equation}
Finally, the proof is finished by combing \cref{eqn:BondBasedTruncationError,eqn:GDcompositionerror}. 
\end{proof}

With stability \cref{thm:stability} and consistency \cref{lem:consistency} of the RK collocation method, we can immediately show the convergence to the nonlocal solution.

\begin{thm} \label{thm:convergencetononlocal}
\textbf{(Uniform Convergence to nonlocal solution)} For a fixed $\delta \in (0, \delta_0]$, assume the nonlocal exact solution $\bm{u}^{\delta}$ is sufficiently smooth, i.e., $\bm{u}^{\delta} \in C^4(\overline{\Omega_{\mathcal{\delta}}}; \, \mathbb{R}^{\nd})$. Moreover, assume $|{\bm{u}^{\delta}}^{(4)}|_{\infty}$ is uniformly bounded for every $\delta$. Let $\bm{u}^{\delta, h}$ be the numerical solution of the collocation scheme \cref{eqn:CollocationScheme}, then,
\begin{equation} \notag 
\| \bm{u}^{\delta} - \bm{u}^{\delta, h} \|_{L^2(\Omega; \, \mathbb{R}^{\nd})} \leq C h_{\max}^2,
\end{equation}
where $C$ is independent of $h_{\max}$ and $\delta$.
\end{thm}

\begin{proof}
First, we can extend $\bm{u}^{\delta}$ to $\R^{\nd}$ by zero such that $\bm{u}^\delta \in C^4(\R^{{\nd}}; \, \mathbb{R}^{\nd})$ because $\bm{u}^{\delta} = \bm{0}$ on $\Omega_{\mathcal{I}}$.  Recall the nonlocal model \cref{eqn:NonlocalEqn} and the collocation scheme \cref{eqn:CollocationScheme},
\[ 
- r^h_{\Omega}\mathcal{L}^S_{\delta} \bm{u}^{\delta,h} = r^h_{\Omega} \bm{f} = - r^h_{\Omega}\mathcal{L}^S_{\delta} \bm{u}^{\delta} \, .
\]
Then, gathering \cref{thm:stability}, \cref{lem:consistency} and the above equation, we have
\begin{align*}
\| \Pi^h \bm{u}^{\delta}  - \bm{u}^{\delta,h}  \|_{L^2(\Omega; \, \mathbb{R}^{\nd})} &    \leq  C \left| r^h_{\Omega} \mathcal{L}^S_{\delta} \left(\Pi^h \bm{u}^{\delta} - \bm{u}^{\delta,h}  \right)  \right|_{h} ,\\
& \leq C \left| r^h_{\Omega} \mathcal{L}^S_{\delta} \Pi^h \bm{u}^{\delta} - r^h_{\Omega} \mathcal{L}^S_{\delta} \bm{u}^{\delta, h}  \right|_{h},\\
& \leq C \left| r^h_{\Omega} \mathcal{L}^S_{\delta} \Pi^h \bm{u}^{\delta} - r^h_{\Omega} \mathcal{L}^S_{\delta} \bm{u}^{\delta}  \right|_{h},\\
& \leq C  h_{\max}^2 \, .
\end{align*}
We finish the proof by applying the triangle inequality
\begin{align*}
\|\bm{u}^{\delta}  - \bm{u}^{\delta,h}  \|_{L^2(\Omega; \, \mathbb{R}^{\nd})} &\leq \| \bm{u}^{\delta}  - \Pi^h \bm{u}^{\delta}  \|_{L^2(\Omega; \, \mathbb{R}^{\nd})}  +  \| \Pi^h \bm{u}^{\delta}  - \bm{u}^{\delta,h}  \|_{L^2(\Omega; \, \mathbb{R}^{\nd})} \\
& \leq C h_{\max}^2.
\end{align*}
\end{proof}


Before showing that the convergence of the RK collocation scheme  to the local limit is independent of $\delta$, we need the bound of truncation error between the collocation scheme and the local limit of the peridyanmic Navier model. 

\begin{lemma} \label{lem:DiscreteME}
\textbf{(Asymptotic consistency I)} Assume $\bm{u} \in C^4(\R^{{\nd}}; \, \mathbb{R}^{\nd})$, then
\begin{equation} \notag 
| r^h\mathcal{L}^S_{\delta}\Pi^h \bm{u} - r^h \mathcal{L}^S_0 \bm{u}|_{h} \leq C |\bm{u}^{(4)}|_{\infty} (h_{\max}^2 + \delta^2) ,
\end{equation}
where $C$ is independent of $h_{\max}$ and $\delta$. 
\end{lemma}
\begin{proof}
From \cref{lem:consistency} and the continuous property of the nonlocal operators, we have
\begin{align*} 
\left| r^h \mathcal{L}^S_{\delta} \Pi^h \bm{u} - r^h \mathcal{L}^S_{0} \bm{u} \right|_{h} & \leq  \left| r^h \mathcal{L}^S_{\delta} \Pi^h\bm{u} - r^h \mathcal{L}^S_{\delta} \bm{u} \right|_{h}  + \left| r^h \mathcal{L}^S_{\delta} \bm{u} - r^h \mathcal{L}^S_{0} \bm{u} \right|_{h}, \\
& \leq C \left|\bm{u}^{(4)} \right|_{\infty} (h_{\max}^2 + \delta^2). 
\end{align*}
\end{proof}

Combining \cref{thm:stability} and \cref{lem:DiscreteME}, we have the uniform convergence (asymptotic compatibility) to the local limit. We leave out the proof of the next theorem for conciseness because it is similar to the proof of \cref{thm:convergencetononlocal}.

\begin{thm} \label{thm:AC}
\textbf{(Asymptotic compatibility)}  Assume the local exact solution $\bm{u}^0$ is sufficiently smooth, i.e., $\bm{u}^0 \in C^4(\overline{\Omega_\delta}; \, \mathbb{R}^{\nd})$. For any $\delta \in(0, \delta_0]$, $\bm{u}^{\delta, h}$ is the numerical solution of the collocation scheme \cref{eqn:CollocationScheme}, then,
\begin{equation} \notag 
\| \bm{u}^0 - \bm{u}^{\delta, h} \|_{L^2(\Omega; \, \mathbb{R}^{\nd})} \leq C (h_{\max}^2 + \delta^2).
\end{equation}
\end{thm}

\section{Convergence analysis of the RK collocation on the quasi-discrete peridynamic Navier equation} \label{sec:DConvergence}
In practice, accurate evaluation of the integral in nonlocal models is computationally prohibitive especially if the nonlocal kernel is singular. This motivates us to use the quasi-discrete nonlocal models as introduced in \cref{subsec:quasi-nonlocal}. It is practical to couple $\delta$ with grid size $h_{\max}$ because this results in a banded linear system. In this section, we assume $\delta = M_0 h_{\max} $ where $M_0>0$.  As $h_{\max}$ goes to zero, so does $\delta$, and the quasi-discrete nonlocal operator converges to its local limit.  We provide convergence analysis of the collocation scheme \cref{eqn:DCollocationScheme} to its local limit. 

\subsection{Stability of the RK collocation on the quasi-discrete peridynamic Navier equation}
We start with the stability of the collocation scheme \cref{eqn:DCollocationScheme}.

\begin{thm} \label{thm:DStability}
For any $\delta \in (0, \delta_0]$, there exists a generic constant $C$ which depends on $\Omega$, $\delta_0$ and $M_0$, such that for $\bm{u} \in S(\square \cap \Omega;\mathbb{R}^{\nd})$, 
\[ |r^h_{\Omega}(-\mathcal{L}^S_{\delta, \epsilon} \bm{u})|_{h} \geq C \| \bm{u} \| _{L^2(\mathbb{R}^{{\nd}}; \, \mathbb{R}^{\nd})}. \]
\end{thm}

To prove \cref{thm:DStability}, we need the Fourier symbol of the quasi-discrete peridynamic Navier operator $\mathcal{L}^S_{\delta, \epsilon}$, shown in \cref{lem:DFPDsymbol}. We present the lemma without proof because the proof follows similarly as \cref{lem:FPDsymbol} using the fact that the quadrature points are symmetric and the quadrature weights are positive \cite{Leng2019b}.

\begin{lemma} \label{lem:DFPDsymbol}
The Fourier symbol of the quasi-discrete peridynamic Navier operator $\mathcal{L}^S_{\delta, \epsilon}$ is given by
\begin{equation} \label{DFPDNoperator}
-\widehat{\mathcal{L}^S_{\delta, {\epsilon}} \bm{u}}(\bm{\xi}) = \bm{M}^S_{\delta, {\epsilon}}(\bm{\xi}) \widehat{\bm{u}} (\bm{\xi}),
\end{equation}
where the Fourier symbol $\bm{M}^S_{\delta, \epsilon}(\bm{\xi})$ is a $\nd \times \nd$ matrix and can be written as

\begin{equation} \label{eqn:DPDNsymbol}
\bm{M}^S_{\delta, \epsilon}(\bm{\xi}) = \bm{M}^B_{\delta, \epsilon}(\bm{\xi}) + \bm{M}^D_{\delta, \epsilon}(\bm{\xi}),
\end{equation}
where 
\begin{equation} 
\bm{M}^B_{\delta, \epsilon}(\bm{\xi}) = \frac{C_{\mu}}{\delta^2}  p^{\epsilon_1}_{1}(\delta|\bm{\xi}|)\left( \bm{I}_{\nd} - \vec{\bm{\xi}} \vec{\bm{\xi}}^{\,\,T} \right) 
+ \frac{C_{\mu}}{\delta^2}   q^{\epsilon_1}_{1}(\delta|\bm{\xi}|)   \vec{\bm{\xi}} \vec{\bm{\xi}}^{\,\,T} ,
\end{equation}
and 
\begin{equation} 
\bm{M}^D_{\delta, \epsilon}(\bm{\xi}) = \frac{C_{\lambda,\mu}}{\delta^2}  \left(b^{\epsilon_1}_{1}(\delta|\bm{\xi}|) \right)^2 \vec{\bm{\xi}} \vec{\bm{\xi}}^{\,\,T} ,
\end{equation}
where the scalars $p^{\epsilon_1}_1( |\bm{\xi}| ), q^{\epsilon_1}_1(|\bm{\xi}| )$ and $b^{\epsilon_1}_1( |\bm{\xi}| )$ are given as follows
\begin{equation} \label{DscalarP1}
 p^{\epsilon_1}_1(|\bm{\xi}| ) = \sum\limits_{\bm{s} \in B^{\epsilon_1}_{1}}  \omega(|\bm{s}|) \rho( |\bm{s}|) \frac{s^2_1}{|\bm{s}|^2}(1-\textnormal{cos}(|\bm{\xi}|s_{\nd})) , 
\end{equation}
\begin{equation} \label{DscalarQ1}
q^{\epsilon_1}_1(|\bm{\xi}| ) = \sum\limits_{\bm{s} \in B^{\epsilon_1}_{1}}  \omega(|\bm{s}|) \rho( |\bm{s}|) \frac{s^2_{\nd}}{|\bm{s}|^2}(1-\textnormal{cos}(|\bm{\xi}|s_{\nd})), 
\end{equation}
\begin{equation} \label{DscalarB1} 
b^{\epsilon_1}_1(\bm{|\xi|}) = \sum\limits_{\bm{s} \in B^{\epsilon_1}_{1}}  \omega(|\bm{s}|) \rho(|\bm{s}|) s_{\nd} \, \textnormal{sin}( |\bm{\xi}|s_{\nd}) .
\end{equation}
\end{lemma}

From \cref{lem:DFPDsymbol}, we have the Fourier representation of the collocation scheme on the quasi-discrete peridynamic Navier operator as follows. 
\begin{lemma} \label{lem:DGCF}
Let $\bm{\widetilde{u}}(\bm{\xi})$ and $\bm{\widetilde{v}}(\bm{\xi})$ be the Fourier series of the sequences $(\bm{u_k}), (\bm{v_k}) \in l^1(\mathbb{Z}^{\nd};\mathbb{C}^{\nd})$ respectively. Then
\begin{equation} \label{eqn:FDCollocation}
((\bm{u_k}), -r^h\mathcal{L}^S_{\delta,{\epsilon}}i^h(\bm{v_k}))_{l^2} = \left(2\pi \right)^{-\nd} \mathlarger{\int}_{\bm{Q}} \bm{\widetilde{u}}(\bm{\xi}) \cdot \bm{M}^{\epsilon}_C(\delta, \bm{h}, \bm{\xi})  \overline{\widetilde{\bm{v}}(\bm{\xi})} d\bm{\xi} , 
\end{equation}
where $\bm{\lambda}^{\epsilon}_C$ are defined as
\begin{equation} \label{eqn:lambdaDC}
\begin{aligned}
\bm{M}^{\epsilon}_C(\delta, \bm{h}, \bm{\xi}) & =  2^{4\nd}\sum_{\bm{r} \in \mathbb{Z}^{\nd}} \bm{M}^B_{\delta, \epsilon}\left((\bm{\xi} + 2 \pi \bm{r}) \oslash \bm{h} \right)  \prod_{j=1}^{\nd}h_j\left(\frac{\sin(\xi_j/2)}{\xi_j+ 2 \pi r_j}\right)^4 \\
& \quad + 2^{4\nd +4}\sum_{\bm{r} \in \mathbb{Z}^{\nd}} \bm{M}^D_{\delta, \epsilon}\left((\bm{\xi} + 2 \pi \bm{r}) \oslash \bm{h} \right)  \prod_{j=1}^{{\nd}}h_j\left(\frac{\sin(\xi_j/2)}{\xi_j+ 2 \pi r_j}\right)^8.
\end{aligned}
\end{equation}
Moreover, there exists $C>0$, independent of $\delta$ and $\bm{h}$ such that,
\begin{equation} \label{eqn:PDofDCC}
\bm{M}^{\epsilon}_C(\delta, \bm{h}, \bm{\xi}) - C \bm{M}_C(\delta, \bm{h}, \bm{\xi})
\end{equation}
is positive definite for any $\bm{\xi} \neq \bm{0}$.
\end{lemma}
\begin{proof} 
The derivation of \cref{eqn:lambdaDC} is similar to \cref{eqn:lambdaC}, we can simply replace $\bm{M}^S_{\delta}(\bm{\xi}+ 2 \pi \bm{r})$ with $\bm{M}^S_{\delta, \epsilon}(\bm{\xi}+ 2 \pi \bm{r})$. The challenge is to show that \cref{eqn:PDofDCC} is positive definite. First, we decompose the set $\bm{Q} = (-\pi, \pi)^{\nd}$ into $\bm{Q}_1$ and  $\bm{Q}_2$, i.e.,
\begin{equation} \notag  
\bm{Q}_1 :=\{ \bm{\xi} \in \bm{Q}: \frac{\delta |\bm{\xi}| }{h_{\min}} \leq \pi \} \textnormal{ and } \bm{Q}_2 = \bm{Q} \backslash \bm{Q}_1,
\end{equation} 
thus for $\bm{\xi} \in \bm{Q}_1$,
\[
\left| \delta s_{\nd} (\bm{\xi} \oslash \bm{h}) \right| \leq \frac{\delta |\bm{\xi}|}{h_{\min}} \leq \pi .
\]
We recall that
\begin{equation} \notag 
p^{\epsilon_1}_1(\delta|\bm{\xi} \oslash \bm{h}| ) = \sum\limits_{\bm{s} \in B^{\epsilon_1}_{1}}  \omega(|\bm{s}|) \rho( |\bm{s}|) \frac{s^2_1}{|\bm{s}|^2}\left(1-\textnormal{cos}\left(\frac{\delta s_{\nd}}{h_{\max}}|\bm{\xi} \oslash \bm{\hat{h}}| \right) \right), 
\end{equation}
and for $x \in (-\pi, \pi)$, there is $C >0 $ such that
\[
1-\textnormal{cos}(x) \geq C x^2;
\]
so for $\bm{\xi} \in \bm{Q}_1$, we obtain

\begin{equation} \notag 
p^{\epsilon_1}_1\left(\delta|\bm{\xi} \oslash \bm{h}| \right)\geq C \left(\frac{\delta |\bm{\xi}| }{h_{\max}}\right)^2 \sum\limits_{\bm{s} \in B^{\epsilon_1}_1}  \omega(|\bm{s}|) \rho( |\bm{s}|) \frac{s^2_1 s^2_{\nd}}{|\bm{s}|^2} \geq C^{\epsilon_1}_p |\bm{\xi}|^2, 
\end{equation}
where we have used the fact that $\bm{h}$ is quasi-uniform ($h_{\max}/h_{\min}$ is bounded above and below), 
\[
\sum\limits_{\bm{s} \in B^{\epsilon_1}_1}  \omega(|\bm{s}|)  \rho( |\bm{s}|) \frac{s^2_1 s^2_{\nd}}{|\bm{s}|^2} >0, \]
is bounded below by \cref{eqn:DBoundedMomentFourth2}, and 
$C^{\epsilon_1}_p $ only depends on $M_0$, $\nd$ and $B^{\epsilon_1}_1$. 
It is easy to notice that if $p^{\epsilon_1}_1(\delta|\bm{\xi} \oslash \bm{h}| ) = 0$, then we must have $M_0 s_{\nd}|\bm{\xi} \oslash \bm{\hat{h}}| = 2 k \pi$ for all $\bm{s} \in B^{\epsilon_1}_1$  for some $k \in \mathbb{Z}^+$. If this happens, we can always add more points $\tilde{\bm{s}} $ to $B^{\epsilon_1}_1$ such that for a certain point $\bm{s}$ in the original set $B^{\epsilon_1}_1$,
$|s_{\nd}|/|\tilde{s_{\nd}}|$  is an irrational number and thus $M_0 \tilde{s_{\nd}}|\bm{\xi} \oslash \bm{\hat{h}}| \neq 2 k \pi$ for any $k \in \mathbb{Z}^+$.
As a consequence, we can choose the set $B_1^{\epsilon_1}$ such that $p^{\epsilon_1}_1(\delta|\bm{\xi} \oslash \bm{h}| )$ is nonzero and $p^{\epsilon_1}_1\left(\delta |\bm{\xi}\oslash \bm{h}| \right) \geq C \geq |\bm{\xi}|^2$ for $\bm{\xi} \in \bm{Q}_2$ because $\bm{Q}_2$ is compact.
Moreover, for $\bm{\xi} \in \bm{Q}$, it is true that
\[
C_1 < \left(\frac{\textnormal{sin}(\xi_j/2)}{\xi_j} \right)^{4} < C_2,
\]
where $C_1, C_2 > 0$ are generic constants.
Therefore, we have for $\bm{\xi} \in \bm{Q}$,

\begin{equation} \label{eqn:Dp1bound}
p^{\epsilon_1}_1\left(\delta \left|\bm{\xi}  \oslash \bm{h} \right| \right) \prod_{j=1}^{{\nd}} h_j \left(\frac{\sin(\xi_j/2)}{\xi_j+  2\pi r_j}\right)^4 
\geq C^{\epsilon_1}_p|\bm{\xi}|^2\prod_{j=1}^{{\nd}} h_j \, . 
\end{equation}
Similarly, we can obtain
\begin{equation} \label{eqn:Dq1bound}
 q^{\epsilon_1}_1\left(\delta |\bm{\xi} \oslash \bm{h}| \right)\prod_{j=1}^{{\nd}} h_j \left(\frac{\sin(\xi_j/2)}{\xi_j+  2\pi r_j}\right)^4  \geq C^{\epsilon_1}_q|\bm{\xi}|^2\prod_{j=1}^{{\nd}} h_j,  
\end{equation}
where $C^{\epsilon_1}_q > 0$ is a generic constant.
Combining \cref{eqn:Dp1bound,eqn:Dq1bound}, we have the following bound, for $\bm{\xi} \in \bm{Q}$,
\begin{equation} \label{eqn:DCFS}
\begin{aligned}
\bm{M}^{\epsilon}_C(\delta, \bm{h}, \bm{\xi}) & \geq C_{\mu} \left(\frac{|\bm{\xi}|}{\delta}\right)^2  \left\{  C^{\epsilon_1}_p  \left( \bm{I}_{\nd} - \vec{\bm{\xi}}_{\bm{h}} \vec{\bm{\xi}}_{\bm{h}}^{\,\,T} \right)   +  C^{\epsilon_1}_q   \vec{\bm{\xi}}_{\bm{h}} \vec{\bm{\xi}}_{\bm{h}}^{\,\,T} \right\}\prod_{j=1}^{{\nd}} h_j \\
& \geq \min \{ C^{\epsilon_1}_p, C^{\epsilon_1}_q \} C_{\mu}  \left(\frac{|\bm{\xi}|}{\delta}\right)^2 \prod_{j=1}^{{\nd}} h_j \, \bm{I}_{\nd} \geq C \left(\frac{|\bm{\xi}|}{\delta}\right)^2 \prod_{j=1}^{\nd} h_j \, \bm{I}_{\nd},
\end{aligned}
\end{equation}
where $\bm{\xi}_{\bm{h}}= \bm{\xi} \oslash \bm{h}$ and we have ignored the terms for $\bm{r} \neq \bm{0}$ because they are non-negative and positive definite.

Next, we use the fact that 
\[
1-\textnormal{cos}(x) \leq x^2 \textnormal{ and } \sin(x) \leq x, \quad \textnormal{for } x\geq 0,
\]
to obtain, for any $\bm{r} \in \mathbb{Z}^{\nd}$,
\begin{equation} \notag 
p_1 \left(\delta |\left(\bm{\xi} + 2\pi \bm{r}\right) \oslash \bm{h}| \right)  \leq  \left(\frac{\delta |\bm{\xi} + 2\pi \bm{r}| }{h_{\max}}\right)^2 \int_{B_{1}} \rho( |\bm{s}|) \frac{s^2_1 s^2_{\nd}}{|\bm{s}|^2}  d\bm{s} \leq C |\bm{\xi} + 2\pi \bm{r}|^2 ,
\end{equation}
and
\[ 
q_1 \left(\delta |\left(\bm{\xi} + 2\pi \bm{r}\right) \oslash \bm{h}| \right)  
\leq C |\bm{\xi} + 2\pi \bm{r}|^2 , \quad 
b_1\left(\delta |\left(\bm{\xi} + 2\pi \bm{r}\right) \oslash \bm{h}| \right)  
\leq C |\bm{\xi}+ 2\pi \bm{r}|,
\]
where we have used \cref{eqn:BoundedMoment}.
Hence we obtain
\begin{equation} \label{eqn:p1bound}
\begin{aligned}
& p_1\left(\delta |\left(\bm{\xi} + 2\pi \bm{r}\right) \oslash \bm{h}| \right) \prod_{j=1}^{\nd} h_j \left(\frac{\sin(\xi_j/2)}{\xi_j+  2\pi r_j}\right)^4 \\
& \leq C  |\bm{\xi}+2\pi\bm{r})|^{2}\left(\frac{\textnormal{sin}(\bm{\xi}/2)}{\bm{\xi} + 2 \pi \bm{r}} \right)^{4} \prod_{j=1}^{\nd} h_j \leq  C_p  \frac{|\bm{\xi}|^2}{ |\bm{\xi}_{\bm{r}}|^2} \prod_{j=1}^{\nd} h_j, \\
\end{aligned}
\end{equation}
where $\bm{\xi}_{\bm{r}}=\bm{\xi}+2\pi\bm{r} $ and $C_p$ is a generic constant. Similarly,
\begin{equation} \label{eqn:q1bound} 
 q_1\left(\delta |\left(\bm{\xi} + 2\pi \bm{r}\right) \oslash \bm{h}| \right)\prod_{j=1}^{\nd} h_j \left(\frac{\sin(\xi_j/2)}{\xi_j+  2\pi r_j}\right)^4  
\leq  C_q \frac{|\bm{\xi}|^2}{ |\bm{\xi}_{\bm{r}}|^2} \prod_{j=1}^{\nd} h_j 
\end{equation}
and 
\begin{equation} \label{eqn:b1bound}
\left[ b_1\left(\delta |\left(\bm{\xi} + 2\pi \bm{r}\right) \oslash \bm{h}| \right) \right]^2 \prod_{j=1}^{\nd} h_j \left(\frac{\sin(\xi_j/2)}{\xi_j+  2\pi r_j}\right)^8  
\leq  C_b \frac{|\bm{\xi}|^2}{ |\bm{\xi}_{\bm{r}}|^2} \prod_{j=1}^{\nd} h_j,
\end{equation}
where $C_q, C_b > 0$. By gathering \cref{eqn:p1bound,eqn:q1bound,eqn:b1bound}, we have 
\begin{equation} \label{eqn:CFS}
\begin{aligned}
&  \bm{M}_C(\delta, \bm{h}, \bm{\xi}) \\ 
& \leq  \left(\frac{|\bm{\xi}|}{\delta}\right)^2  \mathlarger{\sum}_{\bm{r}\in \mathbb{Z}^{\nd}} \frac{C_p C_{\mu}  \left( \bm{I}_{\nd} - \vec{\bm{\xi}}_{\bm{h},\bm{r}} \vec{\bm{\xi}}_{\bm{h,r}}^{\,\,T} \right) +  C_q  C_{\mu}  \vec{\bm{\xi}}_{\bm{h},\bm{r}} \vec{\bm{\xi}}_{\bm{h},\bm{r}}^{\,\,T} +  C_b C_{\lambda, \mu} \vec{\bm{\xi}}_{\bm{h},\bm{r}} \vec{\bm{\xi}}_{\bm{h},\bm{r}}^{\,\,T}}{|\bm{\xi}_{\bm{r}}|^2} \prod_{j=1}^{\nd} h_j, \\
& \leq  C\left(\frac{|\bm{\xi}|}{\delta}\right)^2 \prod_{j=1}^{\nd} h_j \, \bm{I}_{\nd},
\end{aligned}
\end{equation}
where $\bm{\xi}_{\bm{h},\bm{r}}=({\bm{\xi}+2\pi\bm{r}}) \oslash \bm{h}$. 

Finally, \cref{eqn:PDofDCC} is shown by combing \cref{eqn:DCFS,eqn:CFS}. 
\end{proof}

\begin{proof}[Proof of \cref{thm:DStability}] 
By applying \cref{lem:DGCF}, the proof follows similarly to the proof of \cref{thm:stability}.
\end{proof}

\subsection{Consistency of the RK collocation on the quasi-discrete peridynamic Navier equation}

Before showing the discrete model error between $\mathcal{L}^S_{\delta, \epsilon}$ and $\mathcal{L}^S_0$, we need the truncation error between $\mathcal{L}^S_{\delta}$ and $\mathcal{L}^S_{\delta, \epsilon}$. 
\begin{lemma} \label{lem:DiffernceNonlocal}
Assume $\bm{u} \in C^4(\mathbb{R}^{\nd};\mathbb{R}^{\nd})$, then for $i= 1, \ldots, \nd$,
\[
\left| \left[\mathcal{L}^S_{\delta, \epsilon} \bm{u} - \mathcal{L}^S_{\delta} \bm{u}\right]_i \right| \leq C\delta^2 \left|\bm{u}^{(4)}\right|_{\infty}.
\]
\end{lemma}
\begin{proof}
Using Taylor's theorem, for $\bm{x} \in \mathbb{R}^{\nd}$, $j = 1, \ldots, \nd$, and $\bm{s} \in B_{\delta}$ we have
\begin{equation}  \label{eqn:ThirdOrderRemainder}
u_j(\bm{x}+\bm{s}) - u_j(\bm{x})  = \sum_{|\bm{\alpha}|=1,2}\bm{s}^{\bm{\alpha}} \frac{D^{\bm{\alpha}}u_j(\bm{x})}{\bm{\alpha}!} + \sum_{|\bm{\beta}|=3} \bm{s}^{\bm{\beta}} \frac{R_j^{\bm{\beta}}(\bm{y})}{\bm{\beta}!} \, .
\end{equation}
and 
\begin{equation} \label{eqn:Remainder}
u_j(\bm{x}+\bm{s}) + u_j(\bm{x}-\bm{s}) - 2u_j(\bm{x})  = 2\sum_{|\bm{\alpha}|=2}\bm{s}^{\bm{\alpha}} \frac{D^{\bm{\alpha}}u_j(\bm{x})}{\bm{\alpha}!} + \sum_{|\bm{\beta|=4}} \bm{s}^{\bm{\beta}} \frac{R_j^{\bm{\beta}}(\bm{y})}{\bm{\beta}!}\,,
\end{equation}
where $|R_j^{\bm{\beta}}(\bm{y})|\leq C |u_j^{(4)}|_\infty$ and $\bm{y}$ depends on $\bm{x}$ and $\bm{s}$. First,  we study the truncation error between $\mathcal{L}^B_{\delta} \bm{u}$ and $\mathcal{L}^B_{\delta, \epsilon} \bm{u} $, for $i=1, \ldots, \nd$,

\begin{equation} \label{eqn:DBondbasedTruncation}
\begin{aligned}
&\left| [\mathcal{L}^B_{\delta, \epsilon} \bm{u}(\bm{x}) - \mathcal{L}^B_{\delta} \bm{u}(\bm{x})]_i \right| \\
& =  \Bigg\rvert \sum^{\nd}_{j=1}   \sum_{|\bm{\alpha}|=2}  \frac{D^{\alpha}u_j(\bm{x})}{\bm{\alpha}!}\left( \sum_{\bm{s} \in B^{\epsilon}_{\delta}} \omega_{\delta}(|\bm{s}|) \rho_{\delta}(|\bm{s}|) \frac{s_i s_j}{|\bm{s}|^2}\bm{s^{\alpha}}  - \int_{B_{\delta}} \rho_{\delta}(|\bm{s}|) \frac{s_i s_j}{|\bm{s}|^2}\bm{s^{\alpha}} d\bm{s} \right) \\
&  + \sum_{|\bm{\beta}|=4} \frac{1}{2\bm{\beta}!}  \left(\sum_{\bm{s} \in B^{\epsilon}_{\delta}} \omega_{\delta}(|\bm{s}|) \rho_{\delta}(|\bm{s}|) \frac{s_i s_j}{|\bm{s}|^2}\bm{s^{\beta}} R^{\bm{\beta}}_j(\bm{y}) - \int_{B_{\delta}} \rho_{\delta}(|\bm{s}|) \frac{s_i s_j}{|\bm{s}|^2}\bm{s^{\beta}} R^{\bm{\beta}}_j(\bm{y}) d\bm{s} \right) \Bigg\rvert \\
& \leq 0 + \left| \bm{u}^{(4)} \right|_{\infty}\sum^{\nd}_{j=1} \sum_{|\bm{\beta}|=4}  \left(\sum_{\bm{s} \in B^{\epsilon}_{\delta}} \omega_{\delta}(|\bm{s}|) \rho_{\delta}(|\bm{s}|) \frac{|s_i s_j|}{|\bm{s}|^2} |\bm{s}|^{\bm{\beta}}  +   \int_{ B_{\delta}} \rho_{\delta}(|\bm{s}|) \frac{|s_i s_j|}{|\bm{s}|^2} |\bm{s}|^{\bm{\beta}}  d\bm{s}  \right) ,\\
& \leq C \delta^2  \left| \bm{u}^{(4)} \right|_{\infty},
\end{aligned}
\end{equation}
where we have used \cref{eqn:WeightedVolume,eqn:DBoundedMoment,eqn:DBoundedMomentFourth2,eqn:Remainder}.

Next, via \cref{eqn:ThirdOrderRemainder}, the quasi-discrete nonlocal divergence operator $\mathcal{D}^{\epsilon}_{\delta}$ acting on $\bm{u}$ can be written as
\begin{equation} \label{eqn:ExpandDDilatation}  
\begin{aligned}
\mathcal{D}^{\epsilon}_{\delta} \bm{u}(\bm{x}) & = \sum^{\nd}_{j=1} \sum_{|\bm{\alpha}|=1,2}  \frac{D^{\bm{\alpha}}u_j(\bm{x})}{\bm{\alpha}!}  \sum_{\bm{s} \in B^{\epsilon}_{\delta}} \omega_{\delta}(|\bm{s}|) \rho_{\delta}(|\bm{s}|)s_j \bm{s^{\alpha}}   \\
& \quad +  \sum^{\nd}_{j=1} \sum_{|\bm{\beta}|=3}  \frac{1}{\bm{\beta}!}  \sum_{\bm{s} \in B^{\epsilon}_{\delta}} \omega_{\delta}(|\bm{s}|) \rho_{\delta}(|\bm{s}|)s_j \bm{s^{\beta}} R^{\bm{\beta}}_j(\bm{y}),  \\
&= \sum_{j=1}^{\nd} u'_j(\bm{x}) +  \sum^{\nd}_{j=1} \sum_{|\bm{\beta}|=3}  \frac{1}{\bm{\beta}!}  \sum_{\bm{s} \in B^{\epsilon}_{\delta}} \omega_{\delta}(|\bm{s}|) \rho_{\delta}(|\bm{s}|)s_j \bm{s^{\beta}} R^{\bm{\beta}}_j(\bm{y}),
\end{aligned}
\end{equation}
where we denote $u_j'(\bm{x}) = \displaystyle{\frac{\nd u_j(\bm{x}) }{\nd{x_j}}}$ and we have used \cref{eqn:DWeightedVolume}.  We immediately have a similar result for $\mathcal{D}_{\delta}$, 
\begin{equation} \label{eqn:ExpandDilatation}  
\mathcal{D}_{\delta} \bm{u}(\bm{x}) =  \sum_{j=1}^{\nd} u'_j(\bm{x}) +  \sum^{\nd}_{j=1} \sum_{|\bm{\beta}|=3}  \frac{1}{\bm{\beta}!}  \int_{ B_{\delta}} \rho_{\delta}(|\bm{s}|)s_j \bm{s^{\beta}} R^{\bm{\beta}}_j(\bm{y})d\bm{s}.
\end{equation}

Then, the truncation error between $\mathcal{G}^{\epsilon}_{\delta} \mathcal{D}^{\epsilon}_{\delta}\bm{u} $ and $\mathcal{G}_{\delta} \mathcal{D}_{\delta}\bm{u}$ is given by

\begin{equation} \label{eqn:DDGCompositionTruncation}  
\begin{aligned}
& \left| [ \mathcal{G}^{\epsilon}_{\delta} \mathcal{D}^{\epsilon}_{\delta}\bm{u} - \mathcal{G}_{\delta} \mathcal{D}_{\delta}\bm{u} ]_i (\bm{x}) \right| \\
& = \Bigg\rvert  \sum^{\nd}_{j=1} \sum_{|\bm{\alpha}|=1,2} \frac{D^{\bm{\alpha}} (u'_j)(\bm{x})}{\bm{\alpha}!} \left( \sum_{\bm{t} \in B^{\epsilon}_{\delta}} \omega_{\delta}(|\bm{t}|) \rho_{\delta}(|\bm{t}|)t_i \bm{t}^{\bm{\alpha}} - \int_{ B_{\delta}}  \rho_{\delta}(|\bm{t}|)t_i \bm{t}^{\bm{\alpha}} d\bm{t}\right)\\
& \quad +   \sum^{\nd}_{j=1} \sum_{|\bm{\gamma}|=3} \frac{1}{\bm{\gamma}!} \left( \sum_{\bm{t} \in B^{\epsilon}_{\delta}} \omega_{\delta}(|\bm{t}|) \rho_{\delta}(|\bm{t}|)t_i \bm{t^{\gamma}} \tilde{R}_j^{\bm{\gamma}}(\bm{z}) - \int_{ B_{\delta}}  \rho_{\delta}(|\bm{t}|)t_i \bm{t^{\gamma}} \tilde{R}_j^{\bm{\gamma}}(\bm{z})  d\bm{t} \right)  \\
& \quad + \sum^{\nd}_{j=1} \sum_{|\bm{\beta}|=3}  \frac{1}{\bm{\beta}!} \sum_{ \bm{t} \in B^{\epsilon}_{\delta}} \omega_{\delta}(|\bm{t}|) \rho_{\delta}(|\bm{t}|)t_i  \sum_{\bm{s} \in B^{\epsilon}_{\delta}} \omega_{\delta}(|\bm{s}|) \rho_{\delta}(|\bm{s}|)s_j \bm{s^{\beta}} \left(R^{\bm{\beta}}_j(\bm{y}+\bm{t}) -R^{\bm{\beta}}_j(\bm{y})\right) \\
& \quad - \sum^{\nd}_{j=1} \sum_{|\bm{\beta}|=3}  \frac{1}{\bm{\beta}!} \int_{ B_{\delta}} \rho_{\delta}(|\bm{t}|)t_i  \int_{B_{\delta}} \rho_{\delta}(|\bm{s}|)s_j \bm{s^{\beta}} \left(R^{\bm{\beta}}_j(\bm{y}+\bm{t}) -R^{\bm{\beta}}_j(\bm{y})\right) d\bm{s}d\bm{t} \Bigg\rvert, \\
&\leq 0 + C  \left(\sum^{\nd}_{j=1} \sum_{|\bm{\gamma}|=3}  \frac{\left|\bm{u}^{(4)}\right|_{\infty}}{\bm{\gamma}!} \sum_{\bm{t} \in B^{\epsilon}_{\delta}} \omega_{\delta}(|\bm{t}|) \rho_{\delta}(|\bm{t}|)|t_i| |\bm{t^{\gamma}}| + \int_{B_{\delta}} \rho_{\delta}(|\bm{t}|)|t_i| |\bm{t^{\gamma}}| d\bm{t} \right)  \\
& \quad + \sum^{\nd}_{j=1} \sum_{|\bm{\beta}|=3}  \frac{\left|\bm{u}^{(4)}\right|_{\infty}}{\bm{\beta}!} \sum_{|\alpha|=1} \sum_{ \bm{t} \in B^{\epsilon}_{\delta}} \omega_{\delta}(|\bm{t}|) \rho_{\delta}(|\bm{t}|) |t_i||\bm{t^{\alpha}}|  \sum_{\bm{s} \in B^{\epsilon}_{\delta}} \omega_{\delta}(|\bm{s}|) \rho_{\delta}(|\bm{s}|)|s_j| \left|\bm{s^{\beta}}\right| \\
& \quad + \sum^{\nd}_{j=1} \sum_{|\bm{\beta}|=3}  \frac{\left|\bm{u}^{(4)}\right|_{\infty}}{\bm{\beta}!} \sum_{|\alpha|=1} \int_{ B_{\delta}} \omega_{\delta}(|\bm{t}|) \rho_{\delta}(|\bm{t}|) |t_i||\bm{t^{\alpha}}| d\bm{t} \int_{B_{\delta}} \omega_{\delta}(|\bm{s}|) \rho_{\delta}(|\bm{s}|)|s_j| \left|\bm{s^{\beta}}\right| d\bm{s}\\
& \leq C \delta^2  \left|\bm{u}^{(4)}\right|_{\infty}. 
\end{aligned}
\end{equation}
where $\tilde{R}_j$ is the remainder by expanding $u'_j$ as  \cref{eqn:ThirdOrderRemainder} and  we have used \cref{eqn:ExpandDDilatation,eqn:ExpandDilatation}. 

\Cref{eqn:DBondbasedTruncation,eqn:DDGCompositionTruncation} together complete the proof.
\end{proof}  
 
Now, we present the discrete model error between the quasi-discrete nonlocal peridynamic Navier equation and its local limit.
\begin{lemma} \label{lem:DConsistency}
\textbf{(Asymptotic consistency II)} Assume $\bm{u} \in C^4(\mathbb{R}^{\nd};\mathbb{R}^{\nd})$, then 
\begin{equation} \notag 
| r^h_{\Omega} \mathcal{L}^S_{\delta, \epsilon}\Pi^h\bm{u} - r^h_{\Omega} \mathcal{L}^S_{0} \bm{u}^0 |_{h} \leq C  \left|\bm{u}^{(4)}\right|_{\infty} (h_{\max}^2 + \delta^2).
\end{equation}
\end{lemma}
\begin{proof}
In order to prove this lemma, we need the following intermediate result
\begin{equation} \label{eqn:DiscreteOnRK}
\left| r^h \mathcal{L}^S_{\delta,\epsilon} \Pi^h{\bm{u}} - r^h \mathcal{L}^S_{\delta,\epsilon} \bm{u} \right|_{h} \leq C h_{\max}^2 \left|\bm{u}^{(4)}\right|_{\infty}.
\end{equation}
The proof of \cref{eqn:DiscreteOnRK} is similar to \cref{lem:consistency}, following the replacement of the nonlocal operators with their quasi-discrete counterparts.
By collecting \cref{eqn:DiscreteOnRK,lem:DiffernceNonlocal}, the discrete model error of collocation scheme \cref{eqn:DCollocationScheme} is given as
\begin{align*}
\left|r^h \mathcal{L}^S_{\delta,\epsilon} \Pi^h\bm{u} - r^h \mathcal{L}^S_{0} \bm{u} \right|_{h} &\leq \left| r^h \mathcal{L}^S_{\delta,\epsilon} \Pi^h\bm{u} - r^h \mathcal{L}^S_{\delta,\epsilon} \bm{u}  \right|_{h}  + \left| r^h \mathcal{L}^S_{\delta,\epsilon}  \bm{u} - r^h \mathcal{L}^S_{\delta} \bm{u} \right|_{h} \\
& \quad + \left| r^h \mathcal{L}^S_{\delta}  \bm{u}- r^h \mathcal{L}^S_{0} \bm{u} \right|_{h}, \\ 
&\leq C \left|\bm{u}^{(4)}\right|_{\infty} (h_{\max}^2 +\delta^2 + \delta^2).
\end{align*}
\end{proof}

Combining \cref{thm:DStability,lem:DConsistency}, we follow similar procedure as the proof \cref{thm:convergencetononlocal} and show that the numerical solution of \cref{eqn:DCollocationScheme}  converges to its local limit. 

\begin{thm}  \label{thm:QAC}
Assume the local exact solution $\bm{u}^0$ is sufficiently smooth, i.e., $\bm{u}^0 \in C^4(\overline{ \Omega};\mathbb{R}^{\nd})$. For any $\delta \in(0, \delta_0]$, let $\bm{u}^{\delta, \epsilon, h}$ be the numerical solution of the collocation scheme \cref{eqn:DCollocationScheme} and fix the ratio between $\delta$ and $h_{\max}$. Then,
\begin{equation} \notag 
\| \bm{u}^0 - \bm{u}^{\delta, \epsilon, h} \|_{L^2(\Omega; \, \mathbb{R}^{\nd})} \leq C (h_{\max}^2 + \delta^2).
\end{equation}
\end{thm}

\section{Numerical example} \label{sec:NumericalExample}
In this section, we validate the convergence analysis in the previous sections by considering a numerical example in two dimension. We let the discretization parameter be $h_1 = 2h_2$ so the collocation grid has $h_{\textnormal{max}} = h_1$. Choosing the manufactured solution $\bm{u}(x_1,x_2) =[ x_1^2(1-x_1)^2+x_2^2(1-x_2)^2, \allowbreak 0 ]^T$, we obtain the right-hand side of \cref{eqn:NonlocalEqn,eqn:LocalEqn} as
\[
\bm{f}_\delta(\bm{x}) = \bm{f}_0(\bm{x}) -
\left[
\displaystyle \frac{18 \lambda }{5}\delta^2, \, 0 \right]^T 
\]
where
\[
\bm{f}_0(\bm{x}) =-\left[ 2\lambda(1-6x_1+6x_2^2)+6\mu(1-4x_1+4x_1^2-2x_2+2x_2^2), \, 0 \right]^T.
\]
We impose the corresponding values of $\bm{u}(\bm{x})$ on $\Omega_{\cI}$ such that the exact value to the local limit matches on $\partial \Omega$. The nonlocal kernel is chosen as $\displaystyle \rho_{\delta}(|\bm{s}|) = \frac{3}{2\pi \delta^3 |\bm{s}|}$, and let $\Omega =  (0,1)^2$, $E=1$ and $\nu =0.4$. Therefore, the Lam\'{e} parameters $ \displaystyle \lambda=E \nu/((1+\nu)(1-2\nu))$ and $\displaystyle \mu = E/(2(1+\nu))$ satisfy the assumption in \cref{lem:PositiveDefinite}. For a fixed $\delta$, we solve the following peridynamic Navier equation
\begin{equation}  \label{eqn:NLEx}
\begin{cases}
-\mathcal{L}^S_{\delta} \bm{u}(\bm{x}) = \bm{f}_\delta(\bm{x}), & \bm{x} \in \Omega, \\
\bm{u}(\bm{x}) = \left[ x_1^2(1-x_1^2)+x_2^2(1-x_2^2), 0 \right]^T, & \bm{x} \in \Omega_{\cI}\,.
\end{cases}
\end{equation}
When $\delta$ goes to zero,
we substitute $\bm{f}_\delta$ with $\bm{f}_0$ in \cref{eqn:NLEx},
and solve the following nonlocal problem 
\begin{equation}  \label{eqn:NTLEx}
\begin{cases}
-\mathcal{L}^S_{\delta} \bm{u}(\bm{x}) = \bm{f}_0(\bm{x}), & \bm{x} \in \Omega, \\
\bm{u}(\bm{x}) = \left[ x_1^2(1-x_1^2)+x_2^2(1-x_2^2), 0 \right]^T, & \bm{x} \in \Omega_{\cI},
\end{cases}
\end{equation}
which converges to the local limit
\begin{equation} \notag  
\begin{cases}
-\mathcal{L}^S_{0} \bm{u}(\bm{x}) = \bm{f}_0(\bm{x}), & \bm{x} \in \Omega, \\
 \quad \quad \bm{u}(\bm{x}) = \bm{0}, & \bm{x} \in  \partial \Omega.
\end{cases}
\end{equation}
We apply the two collocation schemes as \cref{eqn:CollocationScheme,eqn:DCollocationScheme} and investigate their convergence properties.

\subsection{RK collocation}
We first use the scheme as described in \cref{eqn:CollocationScheme} to solve \cref{eqn:NLEx} for a fixed $\delta$ and investigate the convergence property to the nonlocal limit. Then we study the convergence of the numerical solution to the local limit by solving \cref{eqn:NTLEx} and letting $\delta$ go to zero.

\Cref{fig:convergence} shows the convergence profiles.  When $\delta$ is fixed, the numerical solution converges to the nonlocal solution at a second-order convergence rate. Then we couple $\delta$ with $h_{\max}$  by letting both $\delta$ and $h_{\textnormal{max}}$ go to zero but at different rates, numerical solutions converge to the local limit. Second-order convergence rates are observed when $\delta$ goes to zero faster ($\delta = h_{\textnormal{max}}^2$) and at the same rate as $h_{\textnormal{max}}$ ($\delta = h_{\textnormal{max}}$).  We only obtain a first-order convergence rate when $\delta=\sqrt{h_{\textnormal{max}}}$. The convergence behaviour agrees with \cref{thm:convergencetononlocal} and \cref{thm:AC} and the numerical examples have verified that the RK collocation method is an AC scheme.

\begin{figure}[H]
\centering
\scalebox{0.5}{\includegraphics{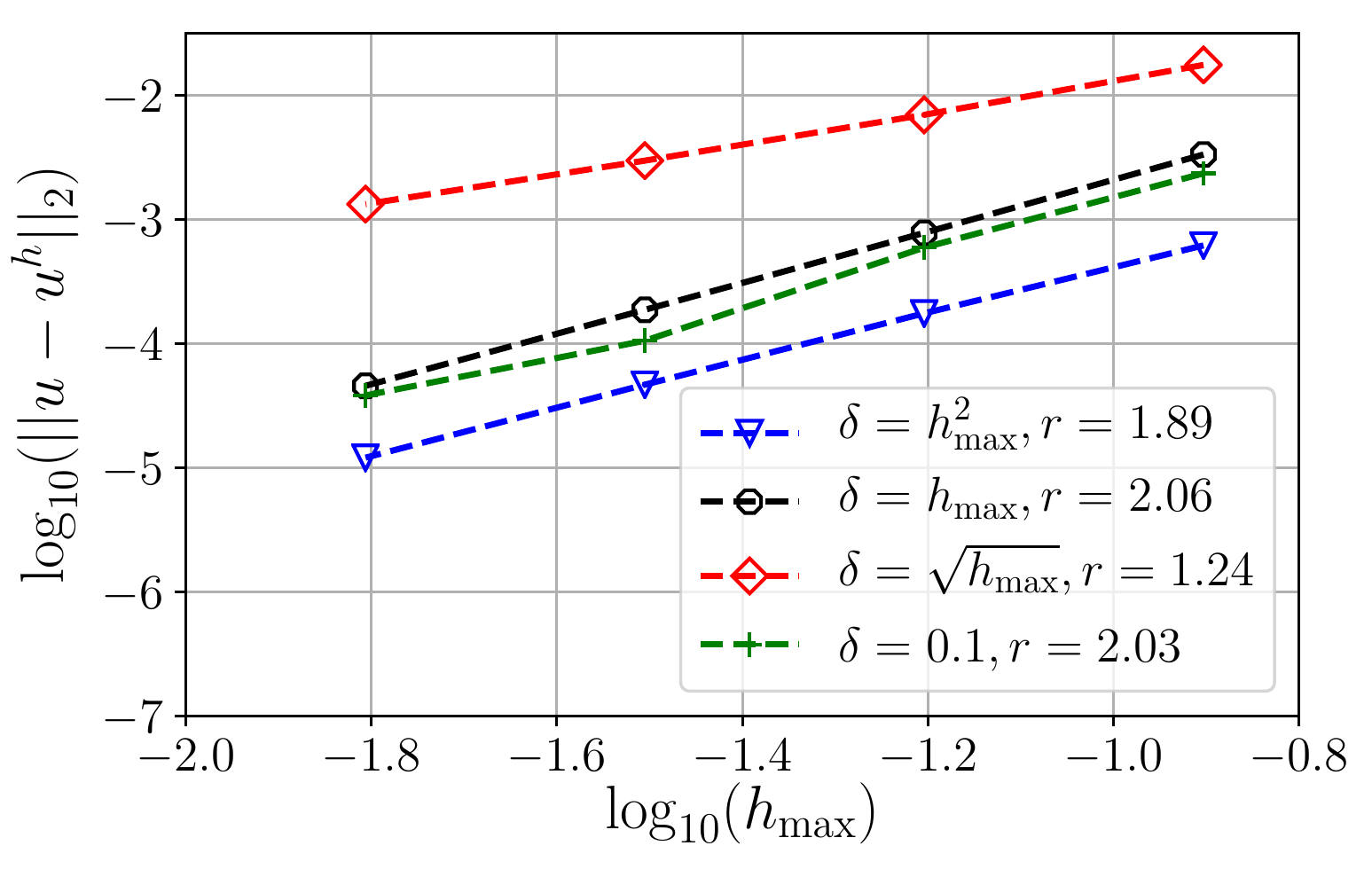}}
\caption{Convergence profiles using the RK collocation method. }
\label{fig:convergence}
\end{figure}

\subsection{RK collocation on quasi-discrete peridynamic Navier equation}
To avoid the need of using high-order Gauss quadrature rules, we have reformulated the peridynamic Navier equation in \cref{subsec:quasi-nonlocal}, using quasi-discrete nonlocal operators. It is also more practical to couple the horizon with grid size as $\delta=M_0 h_{\max}$ because this leads to banded linear systems amenable to traditional preconditioning techniques.  Now, we use the RK collocation method on the quasi-discrete peridynamic Navier equation as discussed in \cref{eqn:DCollocationScheme} to solve \cref{eqn:NTLEx} and study the convergence to the local limit because $\delta$ and $h_{\max}$ approach to $0$ at the same rate. \Cref{fig:MFconvergence} presents the convergence profiles and second-order convergence rates are observed. The numerical findings agree with our analysis in \cref{thm:QAC} and verify that the RK collocation on quasi-discrete peridynamic Navier equation converges to the correct local limit. 

\begin{figure}[H]
\centering
\scalebox{0.5}{\includegraphics{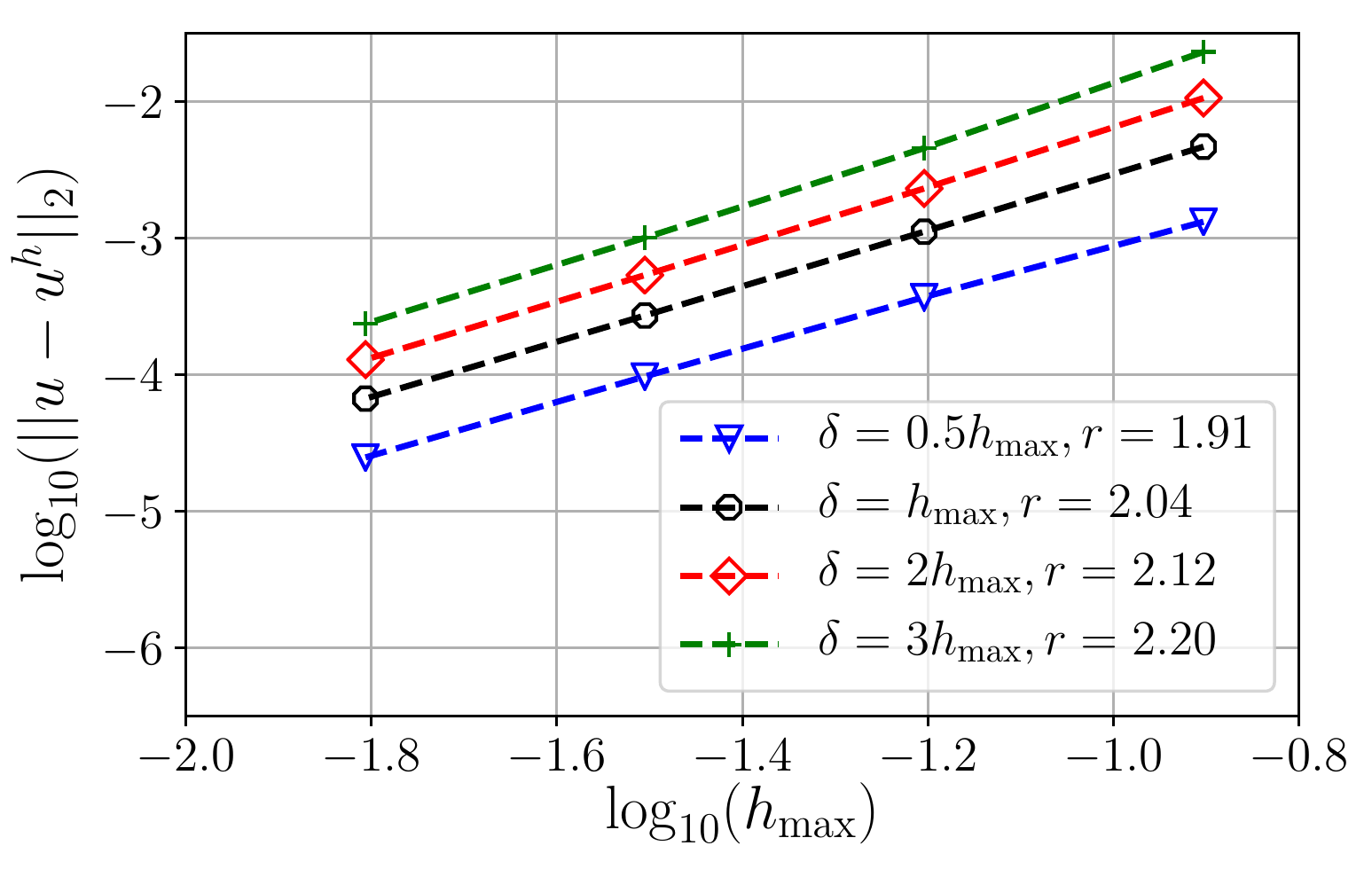}}
\caption{Convergence profiles using the RK collocation on quasi-discrete peridynamic Navier equation. }
\label{fig:MFconvergence}
\end{figure}

\section{Conclusion} \label{sec:Conclusion}
In this work, we have extended a previously developed linear RK collocation method to the peridynamic Navier equation. We first apply linear RK approximation to both the displacements and dilatation, then back-substitute dilatation into the equation, and solve it in a pure displacement form.
Numerical solutions of the method converge to both the nonlocal solution when $\delta$ is fixed and its local limit when $\delta$ vanishes; convergence analysis of this scheme is presented in the case of Cartesian grids with varying resolution in each dimension. Because the standard Galerkin scheme has been proven to be stable, the key idea of analyzing the stability of the collocation scheme was to establish a relationship between the two schemes. When proving stability, in order to avoid constraining the nonlocal kernel, we also assume the material parameters satisfy $\lambda \geq \mu$, and our analysis is applicable for materials with Poisson ratio between $[0.25, 0.5)$.

Then, we formulated the quasi-discrete version of the peridynamics Navier equation using the quasi-discrete nonlocal operators which were proposed in \cite{Leng2019b}. The key was to replace the integral with a finite number of symmetric quadrature points in the horizon with carefully designated quadrature weights satisfying polynomial reproducing conditions for a given nonlocal (even singular) kernel.
Under the assumption that the quadrature points are symmetrically distributed and that the quadrature weights are positive, we have shown the stability of the RK collocation method on the quasi-discrete peridynamics Navier equation.  
The numerical solution of the RK collocation method applied to the quasi-discrete peridynamic Navier equation converges to the correct local limit. 

We have faced two main challenges in this work, comparing to the previous work in \cite{Leng2019b}. The first challenge is the derivation of the Fourier symbol. The Fourier symbol of the peridynamic Navier operator is a matrix and consists of two parts, while the Fourier symbol of the nonlocal diffusion is a scalar; more involved derivations are done for the Fourier representations of the collocation schemes of the peridynamic Navier operator and its quasi-discrete counterpart. The other challenge is the design of the quadrature weights for the quasi-discrete nonlocal operators. A reformulation of the bounded second-order moment condition is required to guarantee consistency.

In addition, we have conducted numerical examples in two dimension to complement our mathematical analysis and observed the same order of convergence as in our theoretical results. That is, for the RK collocation method, the numerical solution converges to the nonlocal solution for a fixed $\delta$ and its local limit independent of the coupling of $\delta$ and discretization parameter $h_{\max}$; for the RK collocation method on the quasi-discrete peridynamic Navier equation, the numerical solution converges to the correct local limit when the ratio $\delta/h_{\max}$ is fixed. 

Finally, we remark that this is the second work of meshfree methods for nonlocal models. Some interesting topics remain to be addressed. For classical (local) linear elasticity, FEM solution obtained from the pure displacement form often deteriorates and becomes unstable when $\nu $ is close to 0.5. For the peridynamic Navier equation, however, numerical results in \cite{Trask2019state} show that the meshfree discretization converges to the local limit with a second-order convergence rate even for $\nu = 0.495$. It is a challenging question to answer, but nevertheless worthwhile, to ask why the peridynamic Navier equation does not have an instability? Moreover, our analysis is limited on rectilinear Cartesian grids but rigorous analysis on a more general grid, such as quasi-uniform grid, should also be studied in the near future.

\section*{Acknowledgements}
The research of Yu Leng and John T. Foster is supported in part by the AFOSR MURI Center for Material Failure Prediction through Peridynamics (AFOSR Grant NO. FA9550-14-1-0073) and the SNL:LDRD academic alliance program. The work of Xiaochuan Tian is supported in part by NSF  grant DMS-1819233. Nathaniel Trask also acknowledges funding under the DOE ASCR  PhILMS center (Grant number DE-SC001924) and the Laboratory Directed Research and Development program at Sandia National Laboratories. Sandia National Laboratories is a multi-program laboratory managed and operated by National Technology and Engineering Solutions of Sandia, LLC., a wholly owned subsidiary of Honeywell International, Inc., for the U.S. Department of Energy’s National Nuclear Security Administration under contract DE-NA-0003525.

The Oden Institute is acknowledged for its support. The authors also thank Leszek Demkowicz, Qiang Du  and Xiao Xu for helpful discussions on the subject. 

\appendix
\gdef\thesection{\Alph{section}}
\makeatletter
\renewcommand\@seccntformat[1]{Appendix \csname the#1\endcsname.\hspace{0.5em}}
\makeatother

\section{} 

\subsection{Proof of \cref{lem:FPDsymbol}} \label{pf:lemFPDsymbol}
We need to calculate the Fourier symbol of the nonlocal operators first. 
\begin{lemma} \label{lem:Fsymbol}
The Fourier symbol of the operators $\mathcal{L}^B_{\delta}, \mathcal{G}_{\delta}, \mathcal{D}_{\delta} $ are given by 
\begin{equation} \label{eqn:FBondbasedOperator}
-\widehat{\mathcal{L}^B_{\delta} \bm{u}}(\bm{\xi}) =  \bm{M}^B_{\delta}(\bm{\xi}) \bm{\widehat{u}} (\bm{\xi}), 
\end{equation}
\begin{equation} \label{eqn:FGradientOperator}
\widehat{\mathcal{G}_{\delta} \theta}(\bm{\xi}) = i \bm{b}_{\delta}(\bm{\xi}) \widehat{\theta}(\bm{\xi}),
\end{equation}
\begin{equation} \label{eqn:FDivergenceOperator}
\widehat{\mathcal{D}_{\delta} \bm{u}}(\bm{\xi}) = i \bm{b}^T_{\delta}(\bm{\xi}) \widehat{\bm{u}}(\bm{\xi}),
\end{equation}
where $\bm{\lambda}_{\delta}(\bm{\xi})$ is a $\nd \times \nd$ matrix and $\bm{b}_{\delta}(\bm{\xi})$ is a vector. They are expressed as 
\begin{equation} \label{eqn:bondsymbol}
\begin{aligned}
\bm{M}^B_{\delta}(\bm{\xi}) &= \int_{B_{\delta}} \rho_{\delta}(|\bm{s}|) \frac{\bm{s} \otimes \bm{s}}{|\bm{s}|^2} (1-\textnormal{cos}(\bm{s}\cdot \bm{\xi})) d\bm{s}, \\
&=p_{\delta}(|\bm{\xi}|) \left( \bm{I}_\nd - \vec{\bm{\xi}} \vec{\bm{\xi}}^{\,\,T} \right)  + q_{\delta}(|\bm{\xi}|) \vec{\bm{\xi}} \vec{\bm{\xi}}^{\,\,T},
\end{aligned}
\end{equation}
and 
\begin{equation} \label{eqn:gradientsymbol}
\bm{b}_{\delta}(\bm{\xi}) = \int_{B_{\delta}} \rho_{\delta}(|\bm{s}|) \bm{s} \, \textnormal{sin}(\bm{s}\cdot \bm{\xi}) d\bm{s} =b_{\delta}(|\bm{\xi}|)\vec{\bm{\xi}} \, ,
\end{equation}
where \( \displaystyle \vec{\bm{\xi}} =\frac{\bm{\xi}}{|\bm{\xi}|} \) is the unit vector in the direction of $\bm{\xi}$ and the scalars $p_{\delta}(|\bm{\xi}|), q_{\delta}(|\bm{\xi}|)$ and $b_{\delta}(|\bm{\xi}|)$ are given by
\begin{equation} \label{eqn:scalarP}
p_{\delta}(|\bm{\xi}|) = \int_{B_{\delta}} \rho_{\delta}(|\bm{s}|) \frac{s^2_1}{|\bm{s}|^2}(1-\textnormal{cos}(|\bm{\xi}|s_{\nd})) d\bm{s} ,
\end{equation}
\begin{equation} \label{eqn:scalarQ}
    q_{\delta}(|\bm{\xi}|) = \int_{B_{\delta}} \rho_{\delta}(|\bm{s}|) \frac{s^2_{\nd}}{|\bm{s}|^2}(1-\textnormal{cos}(|\bm{\xi}|s_\nd)) d\bm{s} ,
\end{equation}
\begin{equation} \label{eqn:scalarB}
b_{\delta}(\bm{|\xi|}) = \int_{B_{\delta}} \rho_{\delta}(|\bm{s}|) s_{\nd} \, \textnormal{sin}(|\bm{\xi}|s_{\nd}) d\bm{s} .
\end{equation}
\end{lemma}

\begin{proof}
The derivations of \cref{eqn:FBondbasedOperator,eqn:FGradientOperator,eqn:FDivergenceOperator} follow directly from the definition of these nonlocal operators. The derivation of $\bm{b}_{\delta}(\bm{\xi})$ can be found in \cite{Tian2018SPH}, and we follow the same strategy to show $\bm{M}^B_{\delta}(\bm{\xi})$,
\begin{equation} \notag
\begin{aligned}
-\widehat{\mathcal{L}^B_{\delta} \bm{u}}(\bm{\xi}) &=-\int_{\mathbb{R}^{3}} e^{-i\bm{x} \cdot \bm{\xi}} \int_{B_{\delta}}\rho_{\delta}(|\bm{s}|)\frac{\bm{s} \otimes \bm{s}}{|\bm{s}|^2}(\bm{u}(\bm{x}+\bm{s})-\bm{u}(\bm{x}))d\bm{s} d\bm{x}, \\
  &=-\int_{B_{\delta}} \int_{\mathbb{R}^{3}} \rho_{\delta}(|\bm{s}|)\frac{\bm{s} \otimes \bm{s}}{|\bm{s}|^2}(\bm{u}(\bm{x}+\bm{s})-\bm{u}(\bm{x})) e^{-i\bm{x} \cdot \bm{\xi}}d\bm{x} d\bm{s} , \\ 
  &=\int_{B_{\delta}} \rho_{\delta}(|\bm{s}|) \frac{\bm{s} \otimes \bm{s}}{|\bm{s}|^2}(1-e^{i\bm{s}\cdot \bm{\xi}})\bm{\widehat{u}} (\bm{\xi}) d\bm{s}, \\ 
  &= \bm{M}^B_{\delta}(\bm{\xi}) \bm{\widehat{u}} (\bm{\xi}),
\end{aligned}
\end{equation}
where $\bm{M}^B_{\delta}(\bm{\xi})$ is given by the first line of \cref{eqn:bondsymbol} and we have used the symmetry of the nonlocal kernel $\rho_{\delta}(|\bm{s}|)$.

We proceed to show the second line of \cref{eqn:bondsymbol} only for $\nd=3$ because the case $\nd=2$ is similar. For any orthogonal matrix $\bm{\mathcal{R}}$, we have 
\begin{equation} \notag 
\bm{M}^B_{\delta}(\bm{\xi}) = \bm{\mathcal{R}}^T\bm{M}^B_{\delta}(\bm{\mathcal{R}}\bm{\xi})\bm{\mathcal{R}}.
\end{equation}
We let $\bm{\mathcal{R}}$ be the orthogonal matrix which rotates $\bm{\xi}$ to be aligned with $\bm{e}$, ($\bm{e}=(0,0,1)^T$ ), as
\[
\bm{\mathcal{R}}\bm{\xi} = |\bm{\xi}|\bm{e}. 
\]
Then $\bm{\mathcal{R}}\bm{ \xi} \cdot \bm{s} = |\bm{\xi}|s_3 $ and  we have
\begin{equation} \notag
\bm{M}^B_{\delta}(\bm{\xi})
= \int_{B_{\delta}} \rho_{\delta}(|\bm{s}|) \frac{1-\textnormal{cos}(|\bm{\xi}|s_3)}{|\bm{s}|^2} \bm{\mathcal{R}}^T\bm{s} (\bm{\mathcal{R}}^T\bm{s})^T d\bm{s}, 
\end{equation}
$\bm{\mathcal{R}}$ is the rotation matrix that rotates $\bm{\xi}$ by an angle of 

\[
\textnormal{arccos} \left(\bm{e} \cdot \frac{\bm{\xi}}{|\bm{\xi}|} \right) = \textnormal{arccos} \left(\frac{\xi_3}{|\bm{\xi}|} \right),
\]
around the axis in the direction of 
\[
\frac{\bm{\xi} \times \bm{e}}{|\bm{\xi} \times \bm{e}|} =\frac{1}{\sqrt{\xi_1^2+\xi_2^2}}(\xi_2, -\xi_1, 0).
\]
$\bm{\mathcal{R}}$ can be explicitly constructed as 
\begin{equation} \label{eqn:RotationMatrix}
\bm{\mathcal{R}} = 
\begin{bmatrix}
\vspace{.2cm}
\displaystyle \frac{\xi_3}{|\bm{\xi}|}+\frac{\xi^2_2}{\xi_1^2+\xi^2_2} \left(1-\frac{\xi_3}{|\bm{\xi}|} \right) & \displaystyle \frac{-\xi_1\xi_2}{\xi_1^2+\xi^2_2}\left(1-\frac{\xi_3}{|\bm{\xi}|}\right) & \displaystyle -\frac{\xi_1}{|\bm{\xi}|}  \\ 
\vspace{.2cm}
\displaystyle \frac{-\xi_1\xi_2}{\xi_1^2+\xi^2_2}\left(1-\frac{\xi_3}{|\bm{\xi}|}\right) & \displaystyle \frac{\xi_3}{|\bm{\xi}|}+\frac{\xi^2_1}{\xi_1^2+\xi^2_2}\left(1-\frac{\xi_3}{|\bm{\xi}|}\right) & \displaystyle -\frac{\xi_2}{|\bm{\xi}|} \\
\displaystyle \frac{\xi_1}{|\bm{\xi}|} & \displaystyle \frac{\xi_2}{|\bm{\xi}|} & \displaystyle \frac{\xi_3}{|\bm{\xi}|} \\
\end{bmatrix}.
\end{equation}
Hence each component of $\bm{M}_{\delta}(\bm{\xi})$ is written as
\begin{equation} \notag
\begin{aligned}
 [\bm{M}^B_{\delta}(\bm{\xi})]_{ik} &=\int_{B_{\delta}} \rho_{\delta}(|\bm{s}|) \frac{1-\textnormal{cos}(|\bm{\xi}|s_3)}{|\bm{s}|^2} \sum_{j=1}^3\mathcal{R}_{ji}s_j \sum_{l=1}^3\mathcal{R}_{lk}s_l  \, d\bm{s}, \\
 &=\int_{B_{\delta}}  \rho_{\delta}(|\bm{s}|) \frac{1-\textnormal{cos}(|\bm{\xi}|s_3)}{|\bm{s}|^2} \sum_{j=1}^3\mathcal{R}_{ji}\mathcal{R}_{jk}s^2_j  \, d\bm{s}, \quad \textnormal{for } i, k = 1, 2, 3, \\
\end{aligned}
\end{equation}
where $\mathcal{R}_{ij}$ is the component of $\bm{\mathcal{R}}$. We can rewrite the Fourier symbol $\bm{M}^B_{\delta}(\bm{\xi})$ as
\begin{equation} \label{eqn:Interlambda}
\bm{M}^B_{\delta}(\bm{\xi}) =\int_{B_{\delta}(\bm{0})} \rho_{\delta}(|\bm{s}|) \frac{1-\textnormal{cos}(|\bm{\xi}|s_3)}{|\bm{s}|^2} \bm{M}(\bm{\xi}, \bm{s}) \, d\bm{s}, 
\end{equation}
where each component of $\bm{M}(\bm{\xi}, \bm{s})$ is given by
\[
M_{ik} =\sum_{j=1}^3R_{ji}R_{jk}s^2_j \, .
\]
From \cref{eqn:RotationMatrix}, we arrive at
\begin{equation} \label{eqn:InterM} 
\begin{aligned}
\bm{M}(\bm{\xi},\bm{s}) &= \frac{1}{|\bm{\xi}|^2}
\begin{bmatrix}
\vspace{.2cm}
(\xi^2_2+\xi^2_3)s^2_1 +\xi^2_1s^2_3 &  \xi_1 \xi_2 (s^2_3 -s^2_1) & \xi_1 \xi_3 (s^2_3 -s^2_1) \\
\vspace{.2cm}
\xi_2 \xi_1 (s^2_3 -s^2_1) &  (\xi^2_1+\xi^2_3)s^2_1 +\xi^2_2s^2_3 & \xi_2 \xi_3 (s^2_3 -s^2_1) \\
\xi_3 \xi_1 (s^2_3 -s^2_1) &  \xi_3 \xi_2 (s^2_3 -s^2_1) & (\xi^2_1+\xi^2_2)s^2_1 +\xi^2_3s^2_3 \, \\
\end{bmatrix}, \\
&= s^2_1\left( \bm{I}_3 - \vec{\bm{\xi}} \vec{\bm{\xi}}^{\,\,T} \right)+ s^2_3 \vec{\bm{\xi}} \vec{\bm{\xi}}^{\,\,T}, 
\end{aligned}
\end{equation}
where we have used the symmetry of the ball and the equivalence of $s_1$ and $s_2$ in the integrand. Substitute \cref{eqn:InterM} into \cref{eqn:Interlambda}, we obtain
the second line of \cref{eqn:bondsymbol}, and $p_{\delta}(|\bm{\xi}|)$ and $q_{\delta}(|\bm{\xi}|)$ as given in \cref{eqn:scalarP,eqn:scalarQ}. 
\end{proof}

With the establishment of the previous lemma, we now can prove \cref{lem:FPDsymbol}.

\begin{proof}[Proof of \cref{lem:FPDsymbol}]
Due to the scaling of the nonlocal kernel \cref{eqn:NonlocalKernelScaling}, we can rewrite $p_{\delta}(|\bm{\xi}|)$, $q_{\delta}(|\bm{\xi}|)$ and $b_{\delta}(|\bm{\xi}|)$ as the following
\begin{equation} \label{eqn:scalingP}  
p_{\delta}(|\bm{\xi}|) = \frac{ p_1(\delta |\bm{\xi}| ) }{\delta^2 } ,
\end{equation}
\begin{equation} \label{eqn:scalingQ} 
q_{\delta}(|\bm{\xi}|) = \frac{ q_1(\delta |\bm{\xi}| ) }{ \delta^2 } , 
\end{equation}
\begin{equation} \label{eqn:scalingB} 
b_{\delta}(\bm{|\xi|}) =\frac{ b_1(\delta |\bm{\xi}|) }{\delta } ,
\end{equation}
where $p_1(\delta |\bm{\xi}| )$, $q_1(\delta |\bm{\xi}| )$ and $b_1(\delta |\bm{\xi}| )$ are given as in \cref{eqn:scalarP1,eqn:scalarQ1,eqn:scalarB1} respectively. Combing 
\cref{eqn:FBondbasedOperator,eqn:FGradientOperator,eqn:FDivergenceOperator}, we arrive at \cref{eqn:FPDNoperator}. Substituting \cref{eqn:scalingP,eqn:scalingQ} into \cref{eqn:bondsymbol}, \cref{eqn:scalingB} into \cref{eqn:gradientsymbol}, we obtain \cref{eqn:PDNsymbol}. 
\end{proof}

\subsection{Proof of \cref{lem:GCF}} \label{pf:lemGCF}
The inverse Fourier transform of $\widehat{\mathcal{L}^S_{\delta} \bm{u}}(\bm{\xi})$ gives
\begin{equation} \notag 
\begin{aligned}
-\mathcal{L}^S_{\delta} \bm{u}(\bm{x})  &= (2\pi)^{-\nd}\int_{\mathbb{R}^{\nd}} e^{i\bm{x}\cdot \bm{\xi}} \bm{M}^S_{\delta}(\bm{\xi}) \widehat{\bm{u}}(\bm{\xi}) d\bm{\xi},
\end{aligned}
\end{equation}
From Parseval's identity, we have
\begin{align*}
((\Psi_{\bm{k}}), -\mathcal{L}^S_{\delta}(\Psi_{\bm{k'}})) &= (2\pi)^{-\nd} \mathlarger{\int}_{\mathbb{R}^{\nd}} \left(\widehat{\Psi_{\bm{k}}}(\bm{\xi}) \right) \overline{ \bm{M}^S_{\delta}(\bm{\xi})} \overline{\left(\widehat{\Psi_{\bm{k'}}}(\bm{\xi})\right)} d\bm{\xi}, \\
&= (2\pi)^{-\nd}  \sum_{j,\,j'=1}^\nd  \mathlarger{\int}_{\mathbb{R}^{\nd}}  e^{i\left(\left(\bm{x_{k'}}\right)-\left(\bm{x_{k}}\right) \right) \cdot \bm{\xi}} \left[\bm{M}^S_{\delta}(\bm{\xi})\right]_{jj'} \widehat{\Psi_{\bm{0}}}^2 (\bm{\xi}) d\bm{\xi},\\
&= \left(2\pi \right)^{-\nd} \sum_{j,\,j'=1}^\nd \mathlarger{\int}_{\bm{Q}}   e^{i(\bm{k'}-\bm{k})\cdot \bm{\xi}}\left[\bm{M}_G(\delta, \bm{h}, \bm{\xi})\right]_{jj'} d\bm{\xi}, 
\end{align*}
where we have used \cref{eqn:RKShape} and 
the Fourier transform of the RK shape function 
\begin{equation} \notag
\widehat{\Psi_{\bm{0}}}(\bm{\xi})= \prod_{j=1}^{\nd} \widehat{\phi \left(\frac{x_j}{2h_j}\right)}(\xi_j) = \prod_{j=1}^{\nd}h_j \left(\frac{\sin(h_j\xi_j/2)}{h_j\xi_j/2}\right)^4,
\end{equation}
where the Fourier transform of the cubic B-spline function is given as
\begin{equation} \notag
\widehat{\phi}(\xi) = \frac{1}{2}\left(\frac{\textnormal{sin}(\xi/4)}{\xi/4}\right)^4.
\end{equation}
Hence, the Galerkin form \cref{eqn:IGalerkin} can be written as
\begin{align*}
&(i^h(\bm{u_{k}}), -\mathcal{L}^S_{\delta}i^h(\bm{v_{k}})) \\
&= (2\pi)^{-\nd} \sum_{j, j'=1}^\nd \sum_{\bm{k}, \bm{k'}\in \mathbb{Z}^\nd} u_{j,\bm{k}} \, \overline{v_{j', \bm{k'}}}  \mathlarger{\int}_{\bm{Q}} e^{i(\bm{k'}-\bm{k}) \cdot \bm{\xi}} \left[\bm{M}_{G}(\delta, h, \bm{\xi})\right]_{jj'} d\bm{\xi},\\
&= \left(2\pi\right)^{-\nd} \sum_{j,j'=1}^\nd \mathlarger{\int}_{\bm{Q}} \widetilde{u_j}(\bm{\xi}) \overline{\widetilde{v_{j'}}(\bm{\xi})} \left[ \bm{M}_G(\delta, \bm{h}, \bm{\xi})\right]_{jj'} d\bm{\xi},
\end{align*}
and we have proved \ref{FGalerkin}. 

Next, we use the same strategy to express the collocation matrix as
\begin{equation} \notag 
\begin{aligned}
-\left[\mathcal{L}^S_{\delta} \left(\Psi_{\bm{k'}} \right) \right]_{j}(\bm{x_k}) &= (2\pi)^{-\nd} \mathlarger{\int}_{\mathbb{R}^{\nd}} e^{i\bm{x_k}\cdot \bm{\xi}} \bm{M}^S_{\delta}(\bm{\xi})  \left(\widehat{\Psi_{\bm{k'}}}(\bm{\xi})\right)  d\bm{\xi}, \\
& = (2\pi)^{-\nd}  \mathlarger{\int}_{\mathbb{R}^{\nd}} e^{i(\bm{x_k}-\bm{x_{k'}})\cdot \bm{\xi}} \bm{M}^S_{\delta}(\bm{\xi}) \left(\widehat{\Psi_{\bm{0}}}(\bm{\xi})\right)d\bm{\xi}, \\
& = \left(2\pi \right)^{-\nd} \sum_{j'=1}^\nd \mathlarger{\int}_{\bm{Q}} e^{i(\bm{k}-\bm{k'})\cdot \bm{\xi}}  \left[\bm{M}_C  (\delta, \bm{h}, \bm{\xi})\right]_{jj'} d\bm{\xi},
\end{aligned}
\end{equation}
then we arrive at the collocation form \cref{eqn:ICollocation} as
\begin{equation} \notag
\begin{aligned}
&\left((\bm{u_{k}}), -r^h\mathcal{L}^S_{\delta}i^h(\bm{v_{\bm{k'}}})(\bm{x_k})\right)_{l^2}\\
&= \left(2\pi \right)^{-\nd} \sum_{j,j'=1}^\nd   \sum_{\bm{k}, \bm{k'} \in \mathbb{Z}^{\nd}}  u_{j,\bm{k}} \, \overline{v_{j',\bm{k'}}} \mathlarger{\int}_{\bm{Q}}  e^{i(\bm{k'}-\bm{k})\cdot \bm{\xi}} \, \overline{\left[\bm{M}_C(\delta, \bm{h}, \bm{\xi}) \right]_{jj'}} \, d\bm{\xi}   , \\
&= \left( 2\pi \right)^{-\nd} \sum_{j,j'=1}^\nd \mathlarger{\int}_{\bm{Q}} \widetilde{u_j}(\bm{\xi}) \overline{\widetilde{v_{j'}}(\bm{\xi})} \left[ \bm{M}_C(\delta, \bm{h}, \bm{\xi}) \right]_{jj'} d\bm{\xi}.    
\end{aligned}
\end{equation}
This finishes the proof of \ref{FCollocation}. In addition, there exists $C>0$, such that,
\begin{equation} \notag 
p_1 \left(\delta|(\bm{\xi} + 2 \pi \bm{r})\oslash\bm{h}| \right)\prod_{j=1}^\nd \left(\frac{\sin(\xi_j)}{(\xi_j+ 2 \pi r_j)}\right)^4 \left(1- C \left(\frac{\sin(\xi_j)}{\xi_j+  2\pi r_j}\right)^4 \right) >0,
\end{equation}
for $\bm{\xi} \in \bm{Q} \backslash \{\bm{0}\} $,  and $\bm{r} \in \mathbb{Z}^{\nd}$; we can also obtain similar estimates for $q_1(\delta|(\bm{\xi} + 2 \pi \bm{r})\oslash\bm{h}|)$ and $b_1(\delta|(\bm{\xi} + 2 \pi \bm{r})\oslash\bm{h}|)$. Following the procedure as in \cref{lem:PositiveDefinite}, we can see \ref{FGCEquivalent} immediately.



\bibliographystyle{elsarticle-harv}
\bibliography{mybib2}

\end{document}